%% file: p8mor.tex
\documentclass[11pt,a4paper]{article}
\usepackage[body={155mm,245mm},top=25mm]{geometry}
\usepackage{bm}
\usepackage{subcaption}
\usepackage{amscd}
\usepackage{mathrsfs}
\usepackage{epsfig}
\usepackage{amsmath}
\usepackage{amssymb}
\usepackage{amsthm}
\usepackage{epsfig}
\usepackage{verbatim}
\usepackage{url}
\usepackage{color}
\usepackage{hyperref}
\include{pspicture}

\def\argmax{\text{argmax}}

\newcommand{\ninseps}[3]{
\begin{figure}[h]
\begin{center}
 \scalebox{#3}{\includegraphics{#1}}
\end{center}

\vspace{-0.65cm}
\caption{\hspace{0.25cm}#2\label{f:#1}}
\end{figure}
}

\newtheorem{theorem}{Theorem}[section]
\newtheorem{lemma}{Lemma}

\newtheorem{remark}{Remark}

\newtheorem{proposition}{Proposition}[section]
\newtheorem*{theorem*}{Theorem}

\usepackage{makeidx}
\usepackage{mathabx}
\usepackage{mathrsfs} 

\begin{document}
\title{Excessive Backlog Probabilities of Two Parallel Queues}
\author{Kamil Demirberk \"{U}nl\"{u}\footnote{Ankara University, Department of Statistics and Middle East Technical University, Institute of Applied Mathematics}, ~~Ali Devin Sezer\footnote{Middle East Technical University, Institute of Applied Mathematics, Ankara, Turkey}}

\maketitle
\begin{abstract}
Let $X$ be the constrained random walk on ${\mathbb Z}_+^2$ with increments $(1,0)$, $(-1,0)$, $(0,1)$ and $(0,-1)$; $X$ represents, at arrivals
and service completions, the lengths of two queues (or two stacks in computer science applications) working in parallel 
whose service and interarrival times are exponentially distributed
with arrival rates $\lambda_i$ and service rates $\mu_i$, $i=1,2$;
we assume $\lambda_i < \mu_i$, $i=1,2$, i.e., $X$ is assumed stable. Without loss of generality we assume $\rho_1 =\lambda_1/\mu_1  
\ge \rho_2 = \lambda_2/\mu_2$. 
Let $\tau_n$ be the first time $X$ hits the line $\partial A_n = \{x \in {\mathbb Z}^2:x(1)+x(2) = n \}$, i.e.,  when the sum of the components of $X$ equals $n$ for the first time. Let $Y$ be the same random walk as $X$ but only constrained on $\{y \in {\mathbb Z}^2: y(2)=0\}$ and its jump probabilities for the first component reversed. Let $\partial B =\{y \in {\mathbb Z}^2: y(1) = y(2) \}$ and let $\tau$ be the first time $Y$ hits $\partial B$.  The probability $p_n = P_x(\tau_n < \tau_0)$ is a key performance measure of the queueing system (or the two stacks) represented by $X$ (if the queues/stacks share a common buffer, then $p_n$ is the probability that this buffer overflows during the system's first busy cycle).  Stability of the process implies that $p_n$ decays exponentially in $n$ when the process starts off the exit boundary $\partial A_n.$ We show that, for $x_n= \lfloor nx \rfloor$, $x \in {\mathbb R}_+^2$, $x(1)+x(2) \le 1$, $x(1) > 0$, $P_{(n-x_n(1),x_n(2))}( \tau <  \infty)$ approximates $P_{x_n}(\tau_n < \tau_0)$ with exponentially vanishing relative error.  Let $r = (\lambda_1 + \lambda_2)/(\mu_1 + \mu_2)$; for $r^2 < \rho_2$ and $\rho_1 \neq \rho_2$, we construct a class of harmonic functions from single and conjugate points on a related characteristic surface for $Y$ with which the probability $P_y(\tau < \infty)$ can be approximated with bounded relative error. For $r^2 = \rho_1 \rho_2$, we obtain the exact formula $P_y(\tau < \infty) = r^{y(1)-y(2)} +\frac{r(1-r)}{r-\rho_2}\left( \rho_1^{y(1)} - r^{y(1)-y(2)} \rho_1^{y(2)}\right).$ 
\end{abstract}


\section{Introduction}
This work concerns the random walk 
$X$ with independent and identically distributed increments 
$\{I_1,I_2,I_3,...\}$,
constrained to remain in ${\mathbb Z}_+^2$:
\begin{align*}
X_0 &=x \in {\mathbb Z}_+^2, ~~~X_{k+1} \doteq X_k + \pi ( X_k,I_k ), k=1,2,3,...\\
\pi(x,v) &\doteq \begin{cases} v, &\text{ if } x +v  \in {\mathbb Z}^2_+, \\
				  0      , &\text{otherwise,}
\end{cases} 
\end{align*}
\begin{align*}
&I_k \in \{(1,0),(-1,0), (0,1), (0,-1)\},
P(I_k = (1,0)) = \lambda_1,\\
&P(I_k = (0,1)) = \lambda_2,
P(I_k = (-1,0)) = \mu_1,
P(I_k = (0,-1)) = \mu_2.
\end{align*}
The dynamics of $X$ are depicted in Figure \ref{f:dyn2tex}.

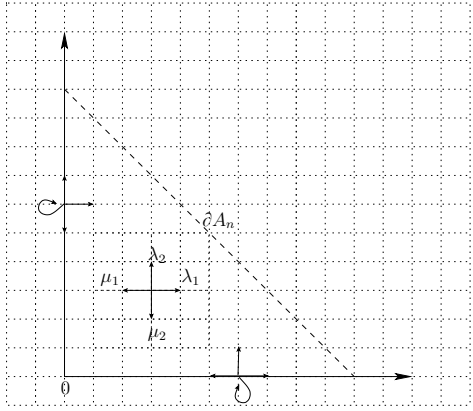
\begin{figure}[h]
\begin{center}
\scalebox{0.6}{
\centerline{\input{dyn2tex}}}
\end{center}

\vspace{-0.65cm}
\caption{\hspace{0.25cm}Dynamics of $X$\label{f:dyn2tex}}
\end{figure}

We denote the constraining boundaries by
$\partial_i \doteq \{ x \in {\mathbb Z}^2: x(i) = 0 \}$, $i=1,2$.
A well known interpretation for $X$ is as the embedded random walk of two parallel queues with Poisson arrivals and independent
and exponentially distributed service times. This random walk appears 
in computer science as a model of two 
stacks running together \cite{knuth1972art,yao1981analysis,flajolet1986evolution,maier1991colliding}.
Define the region
\begin{equation}\label{e:defA}
A_n = \left\{ x \in {\mathbb Z}_+^2: x(1) + x(2) \le n \right\}
\end{equation}
\index{$A_n$}
and its boundary
\begin{equation}\label{e:defBoundaryA}
\partial A_n =
  \left\{ x \in {\mathbb Z}_+^2: x(1) + x(2) = n \right\}.
\end{equation}
Let $\tau_n$ be the first time $X$ hits $\partial A_n$:
\begin{equation}\label{d:taun}
\tau_n \doteq \inf\{k: X_k \in \partial A_n\}.
\end{equation}
When $X$ is stable, i.e., when $\lambda_i  <\mu_i$, a well known performance measure
associated with $X$ is the probability
$p_n \doteq P_x( \tau_n < \tau_0)$; if the queues
share a common buffer, then $p_n$ is the probability that this buffer overflows during
the system's first busy cycle. The stability assumption implies that $p_n$ decays exponentially
in $n$, when the walk starts away from the exit boundary $\partial A_n$. 
Hence $\{\tau_n < \tau_0\}$ is a rare event for $n$ large.
There is no closed form formula for $p_n$ in terms of the parameters of the problem; the approximation of probabilities of the type $p_n$ and of related probabilities and expectations for constrained processes has been a challenge for many years and there is a wide literature on this subject using techniques including large deviations analysis and rare event simulation, see, \cite{ alanyali1998large, aldous2013probability, anantharam, asmussen2007stochastic, asmussen2008applied, atar1999large,  blanchet2006efficient, blanchet2009lyapunov, blanchet2009rare, blanchet2013optimal, BoerNicola02, borovkov2001large, Changetal, collingwood2011networks, comets2007distributed, comets2009large, dai2011reflecting, de2006analysis, dean2009splitting, dieker2005asymptotically, duphui-is1, dupuis1995large, flajolet1986evolution, foley2005large, foley2012constructing, frater1991optimally, guillotin2006dynamic, Iglehart74, ignatiouk2010martin, ignatiouk2000large, ignatyuk1994boundary, juneja2006rare, JunejaNicola, JunejaRandhawa, KDW, kobayashi2013revisiting, KroeseNicola, kurkova1998martin, louchard1991probabilistic, louchard1994random, maier1991colliding, maier1993large, majewski2008large,mcdonald1999asymptotics, miretskiy2008state, miyazawa2009tail, miyazawa2011light, MR1110990, MR1335456, MR1736592, ney1987markov, nicola2007efficient, parekh1989quick, ridder2009importance, Rubetal04, rubino2009rare, sezer2009importance, sezer-asymptotically, yao1981analysis, yeniDW}. 

The goal of the present work is to prove, for the constrained
random walk $X$ defined above, that the approximation approach used in
\cite{sezer2015exit,tandemsezer} gives approximations of $p_n$ with exponentially decaying relative error
 when the initial position $x$ of the random walk is off the boundary $\partial_1.$
The work \cite[Section 6]{tandemsezer} reviews many of the works cited above and points out the relations between the approach of the current paper and \cite{sezer2015exit, tandemsezer} 
and the range of methods and approaches used in these works.  

The approximation technique and the proof method
are reviewed below (see subsection \ref{ss:summary}). Although the main approach of the present work is parallel to  that of
\cite{sezer2015exit, tandemsezer}, new challenges and ideas appear in the treatment of the
present case; there are also differences in the assumptions
made and the results obtained.
These are reviewed in Section \ref{s:compare}.

First, several definitions; the utilization rates of the nodes are:
\[
\rho_i = \frac{\lambda_i}{\mu_i}, i=1,2.
\]
We assume that $X$ is stable, i.e.,
\[
\rho_1, \rho_2 < 1.
\]

The following quantity plays a central role in our analysis:
\[
r = \frac{\lambda_1 + \lambda_2}{\mu_1 + \mu_2}.
\]
Without loss of generality we can assume
\begin{equation}\label{as:orderrho}
\rho_2 \le r \le \rho_1
\end{equation}
(if this doesn't hold, rename the nodes). 
We will make two further technical assumptions:
\begin{equation}\label{as:tech}
\rho_1 \neq \rho_2, \frac{r^2}{\rho_2} < 1.
\end{equation}
The first of these is needed in the construction of the $Y$-harmonic
functions in Section \ref{s:Yharmonic}, see \eqref{d:boldh}.
The second is useful
both in the computation of $P_y(\tau < \infty)$ (see the proof of
Proposition \ref{p:explicit}) and in the limit analysis
(see the proof of Proposition \ref{p:laplace}). We further
comment on these assumptions in the Conclusion (Section \ref{s:conclusion}).

Define
the linear transformation
\[
{\mathcal I} \doteq \left( \begin{matrix}  -1 & 0 \\
                                             0  & 1
			      \end{matrix}
\right)
\]
and the affine transformation
\[
T_n = n e_1 + {\mathcal I}
\]
where $(e_1,e_2)$ is the standard basis for ${\mathbb R}^2.$
Furthermore, define the constraining map
\[
\pi_1(x,y) = 
\begin{cases} y, &\text{ if } x + y \in {\mathbb Z}\times {\mathbb Z}_+,\\
	      0, &\text{otherwise.}
\end{cases}
\]
Define $Y$ to be a constrained random walk on 
${\mathbb Z} \times {\mathbb Z}_+$ 
with increments 
\begin{equation}\label{d:Jk}
J_k \doteq {\mathcal I} I_k:
\end{equation}
\[
Y_{k+1} = Y_k  + \pi_{1}(Y_k, J_k).
\]
$Y$ has  the same
increments as $X$, but the probabilities of the increments $e_1$ and $-e_1$ are reversed.
Define 
\[
\partial B  \subset {\mathbb Z}\times {\mathbb Z}_+, \partial B \doteq \left\{y : y(1) = y(2) \right\}, 
\]
and the hitting time $\tau \doteq \inf\left\{k: Y_k \in \partial B \right\}.$

\subsection{Summary of our analysis}\label{ss:summary}
\cite[Proposition 3.1]{sezer2015exit} asserts,
in a more general framework than the model given above,
that for any $y \in {\mathbb Z}_+^2$, $y(1) > y(2)$,
$P_{T_n(y)}(\tau_n < \tau_0) \rightarrow P_y(\tau < \infty)$.
The approximation idea connecting these two probabilities is
shown in Figure \ref{f:trntex}: by applying $T_n$,
we move the origin of the coordinate
system to $(n,0)$ and take limits, which leads to the limit problem
of computing $P_y(\tau < \infty)$ where the limit $Y$ process is
the same process as $X$ (observed from the point $(n,0)$)
but not constrained on 
$\partial_1.$

\begin{figure}[h]
\begin{center}
\scalebox{0.45}{
\centerline{\input{trntex}}}
\end{center}

\vspace{-0.65cm}
\caption{\hspace{0.25cm}Transformations and the limit problem\label{f:trntex}}
\end{figure}
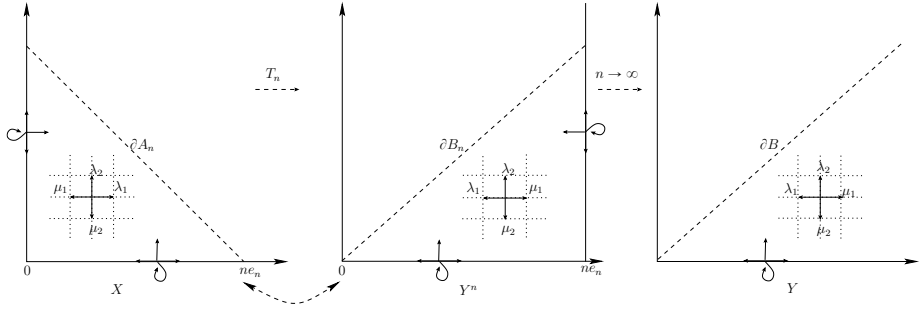

A more interesting
convergence analysis is when the initial point is given in
$x$ coordinates. A convergence analysis from this point of view has only
been performed so far for the constrained random walk with increments
$(1,0)$, $(-1,1)$ and $(0,-1)$ (representing two tandem queues) 
in \cite{sezer2015exit, tandemsezer}. The goal of the present work is to extend this analysis to
the simple random walk $X$.
Our main result is the following theorem:
\begin{theorem*}[Theorem \ref{t:mainapprox}]
For any
$x \in {\mathbb R}_+^2$,
$x(1) + x(2) < 1$, $x(1) > 0$,
there exists $C_7 > 0$ and $N > 0$ such that
\[
\frac{ |P_{x_n}(\tau_n < \tau_0) - P_{T_n(x_n)}(\tau <  \infty)|}{
P_{x_n}(\tau_n < \tau_0)} < e^{-C_7 n}
\]
for $n > N$, where $x_n = \lfloor x n \rfloor$.
\end{theorem*}

Thus, as $n$ increases $P_{T_n(x_n)}(\tau < \infty)$ 
approximates $P_{x_n}(\tau_n < \tau_0)$ very well (with exponentially
decaying relative error in $n$) if $x(1) > 0.$
In the tandem case there is a simple explicit formula for $P_y(\tau < \infty)$.
In the parallel walk
case a simple explicit formula exists under the additional condition
\begin{equation}\label{c:harmonicmean}
\rho_1 \rho_2 = r^2.
\end{equation}
The formula for $P_y(\tau < \infty)$ under this condition is
\begin{equation}\label{e:formula1}
P_y(\tau < \infty) = r^{y(1)-y(2)} + \frac{(1-r)r}{r-\rho_2}
\left( \rho_1^{y(1)} - r^{y(1)-y(2)} \rho_1^{y(2)} \right).
\end{equation}
This is derived in Proposition \ref{p:explicit} and is based on the
class of $Y$-harmonic functions constructed 
in Section \ref{s:Yharmonic}
from single and conjugate
points on a characteristic surface associated with $Y$.
A generalization of \eqref{e:formula1} can be used 
to find upper and lower bounds for $P_y(\tau < \infty)$ 
when \eqref{c:harmonicmean} doesn't hold, see Propositions \ref{p:relativeerr}
and \ref{p:improvedapprox}. Subsection \ref{ss:finer}
illustrates how one can use these results to construct
finer approximations of $P_y(\tau < \infty)$ with diminishing relative
error using superposition of $Y$-harmonic functions defined by single
and conjugate points on the characteristic surface.

Define the stopping times
\begin{equation}\label{d:sigmas}
\sigma_1 = \inf\{k: X_k \in \partial_1\}, ~~~
\bar{\sigma}_1 = \inf \{k: T_n(Y_k) \in \partial_1 \}.
\end{equation}
If we set the initial position of $Y$  to $Y_0 = T_n(X_0)$, we have
\[
\{ \tau_n < \tau_0\} \cap \{ \tau_n < \sigma_1 \wedge \tau_0 \} = 
\{\tau < \infty\} \cap \{\tau < \bar{\sigma}_1 < \infty\}.
\]
The main argument in the proof of Theorem \ref{t:mainapprox} 
is this: most of the probability of the events
$\{ \tau_n < \tau_0\}$ and  $\{\tau < \infty\}$ come from the events 
$\{ \tau_n < \sigma_1 \wedge \tau_0 \}$ and  $\{\tau < \bar{\sigma}_1 < \infty\}$ respectively, if the initial position
$X_0$ of $X$ is away from $\partial_1.$ The full implementation
of this argument will require the following steps:
\begin{enumerate}
\item Construction of $Y$-harmonic functions, $Y-z$ harmonic
functions and  bounds on 
${\mathbb E}_{y}[z^{\tau} 1_{\{ \tau < \infty\}}]$ for $z > 1$ 
(Sections \ref{s:Yharmonic} and \ref{s:laplace}),
\item Large deviations (LD) analysis of $P_{x_n}(\tau_n < \tau_0)$ (Section \ref{s:LD1}),
\item LD analysis of $P_{x_n}( \sigma_1 < \tau_n < \tau_0)$ (Section \ref{s:LD2}),
\item LD analysis of $P_{x_n} (\bar{\sigma}_1 < \tau < \infty)$ 
(Subsection \ref{ss:LD3}).
\end{enumerate}
These steps are put together in Section \ref{s:together}. 
Section \ref{s:Pty} treats the problem of computing
$P_y(\tau < \infty)$ from the $Y$-harmonic functions of Section \ref{s:Yharmonic}.
Section \ref{s:compare} points out the parallels and differences between
the analysis of the constrained walk $X$ treated in the present
work and the tandem walk treated in 
\cite{sezer2015exit,tandemsezer}. 
We comment on future work in the conclusion
(Section \ref{s:conclusion}).

\section{Harmonic functions of $Y$}\label{s:Yharmonic}
A function $h$ on ${\mathbb Z}\times {\mathbb Z}_+$ 
is said to be $Y$-harmonic if
\[
{\mathbb E}_y[h(Y_1)]= h(y), y \in {\mathbb Z}\times {\mathbb Z}_+.
\]
Following \cite{sezer2015exit}, introduce the 
the interior characteristic polynomial of $Y$:
\[
{\bm p}(\beta,\alpha) \doteq
 \lambda_1\frac{1}{\beta}+ \mu_1\beta 
+  \lambda_2 \frac{\alpha}{\beta} +
\mu_2 \frac{\beta}{\alpha}.
\]
and  characteristic polynomial of $Y$ on $\partial_1$:
\[
{\bm p}_1(\beta,\alpha) \doteq
 \lambda_1\frac{1}{\beta}+ \mu_1\beta 
+  \lambda_2 \frac{\alpha}{\beta} +
\mu_2 
\]
As in \cite{sezer2015exit}, 
we will construct $Y$-harmonic functions from solutions of ${\bm p} = 1$; 
the set of all solutions of this
equation defines the characteristic surface
\[
{\mathcal H} \doteq \{ (\beta,\alpha) \in {\mathbb C}^2: 
{\bm p}(\beta,\alpha) =1  \},
\]
define, similarly, the characteristic surface for $\partial_1$:
\[
{\mathcal H}_1 \doteq \{ (\beta,\alpha) \in {\mathbb C}^2: 
{\bm p}_1(\beta,\alpha) =1  \},
\]

Multiplying both sides of ${\bm p}=1$ by $\alpha$
 transforms it to the quadratic equation
\begin{equation}\label{e:pasalpha}
\alpha\left( \lambda_1\frac{1}{\beta}+ \mu_1\beta -1\right)
+  \lambda_2 \frac{\alpha^2}{\beta} +
\mu_2 \beta  = 0,
\end{equation}
Define
\begin{equation}\label{e:conjugator}
\boldsymbol \alpha(\beta,\alpha) \doteq 
\frac{1}{\alpha}\frac{\beta^2}{\rho_2};
\end{equation}
if for a fixed $\beta$, $\alpha_1$ and $\alpha_2$ are distinct roots of \eqref{e:pasalpha},
they will satisfy 
\[
\alpha_2  =\boldsymbol \alpha(\beta,\alpha_1)
\]
by simple algebra;
we will call the points $(\beta,\alpha_1) \in {\mathcal H}$ and $(\beta,\alpha_2) \in  {\mathcal H}$ 
arising from such roots conjugate.
Following \cite{sezer2015exit} we refer to the function $\boldsymbol \alpha$ as the conjugator.
An example of two conjugate points are shown in Figure \ref{f:charsurf0}. 

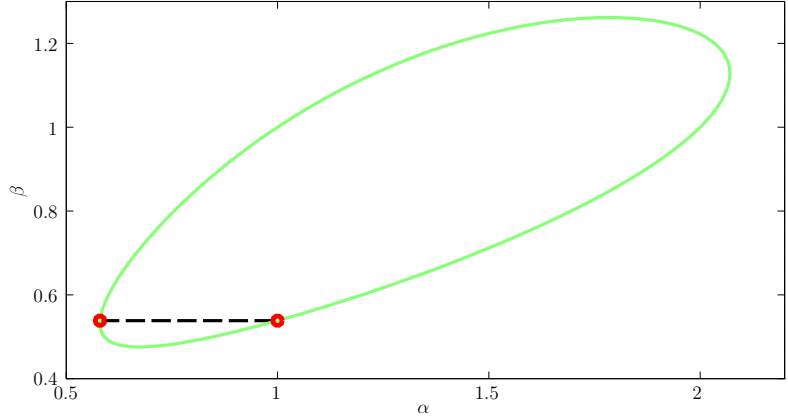
\begin{figure}[h]
\centering
\scalebox{0.6}{
\input{charsurf0}
}
\vspace{-1.5cm}
\caption{The real section of the characteristic surface ${\mathcal H}$
for $\lambda_1 =0.15$, $\lambda_2 = 0.2$,  $\mu_1 = 0.25$, $\mu_2 = 0.4$;
the end points of the dashed line are an example
of a pair of conjugate points $(\beta,\alpha_1)$
and $(\beta,\alpha_2)$; 
Each such pair defines a $Y$-harmonic function,
see Proposition \ref{p:conjugate}}
\label{f:charsurf0}
\end{figure}

For any point $(\beta, \alpha) \in {\mathcal H}$ define the following 
${\mathbb C}$-valued function on ${\mathbb Z}^2$:
\[
z \mapsto [(\beta,\alpha), z], z \in {\mathbb Z}^2,
\]
\[
[(\beta,\alpha),z] \doteq \beta^{z(1)-z(2)} \alpha^{z(2)}.
\]

\begin{lemma}\label{l:interior0}
$[(\beta,\alpha),\cdot]$ is $Y$-harmonic on ${\mathbb Z}\times {\mathbb Z}_+ - \partial_2$
when $(\beta,\alpha) \in {\mathcal H}$. In addition $ x\mapsto [(\beta,\alpha), T_n(x) ]$, $x
\in {\mathbb Z}^2_+$, is $X$-harmonic on ${\mathbb Z}_+^2 - \partial_1 \cup \partial_2.$
\end{lemma}
\begin{proof}
As in \cite{sezer2015exit}, the first claim follows from the definitions involved:
\[
{\mathbb E}_z[ (\beta,\alpha), Z_1 ] = \beta^{z_1 -z_2} \alpha^{z_2} {\bm p}(\beta,\alpha)
= [ (\beta,\alpha), z].
\]
and the second
claim follows from the first and the fact that $J_k = {\mathcal I} I_k$ 
(see \eqref{d:Jk}).
\end{proof}

Define
\[
C(\beta,\alpha)  \doteq \left( 1- \frac{\beta}{\alpha} \right),
(\beta,\alpha) \in {\mathbb C}^2, \alpha \neq 0.
\]
Proceeding parallel to \cite{sezer2015exit},
one can define the following class of $Y$-harmonic functions from
the functions $[(\beta,\alpha),\cdot]$:
\begin{proposition}\label{p:both}
Suppose $(\beta,\alpha) \in {\mathcal H} \cap {\mathcal H}_1$.
Then $[(\beta,\alpha),\cdot]$ is $Y$-harmonic.
\end{proposition}
\begin{proof}
Lemma \ref{l:interior0} says that for $(\beta,\alpha) \in {\mathcal H}$, 
$[(\beta,\alpha),\cdot]$ is $Y$-harmonic on ${\mathbb Z} \times {\mathbb Z}_+ - \partial_2.$
An argument parallel to the one given in the proof of Lemma \ref{l:interior0} 
shows that
$[(\beta,\alpha),\cdot]$ is $Y$-harmonic on $\partial_2$ when $(\beta,\alpha) \in {\mathcal H}_1.$
These two facts imply the statement of the proposition.
\end{proof}

The next proposition gives us another class of $Y$-harmonic functions
constructed from conjugate points on ${\mathcal H}$, it is
a special case of \cite[Proposition 4.9]{sezer2015exit}:
\begin{proposition}\label{p:conjugate}
Suppose $(\beta,\alpha_1) \neq (\beta,\alpha_2)$,  are conjugate points
on ${\mathcal H}$.
Then
\[
h_\beta \doteq C(\beta,\alpha_2) [(\beta,\alpha_1),\cdot]- C(\beta,\alpha_1)
[(\beta,\alpha_2),\cdot]
\]
is $Y$-harmonic.
\end{proposition}
For sake of completeness and easy reference, let us reproduce the argument given in
the proof of \cite[Proposition 4.9]{sezer2015exit}:
\begin{proof}
That $h_\beta$ is $Y$-harmonic on ${\mathbb Z} \times {\mathbb Z}_+ -\partial_2$
follows from Lemma \ref{l:interior}.  For $y \in \partial_2$
a direct computation gives:
\[
{\mathbb E}_y[[(\beta,\alpha_i), y+\pi_1(y,J_1)]] -[(\beta,\alpha_i),y] = 
\mu_2 C(\beta,\alpha_i) \beta^{y(1)}.
\]
It follows that
\[
{\mathbb E}_y[h_\beta( y+\pi_1(y,J_1)] -h_\beta(y)
= \mu_2C(\beta,\alpha_1)C(\beta,\alpha_2) (\beta^{y(1)} - \beta^{y(1)}) = 0,
\]
i.e., $h_\beta$ is $Y$-harmonic on $\partial_2$ as well.
\end{proof}

The intersection of 
${\mathcal H}$ and ${\mathcal H}_1$ consists
of the points
$(0,0)$, $(1,1)$ and $(\rho_1,\rho_1)$. The last of these gives us our first
nontrivial loglinear $Y$-harmonic function:
\begin{lemma}\label{l:rho1}
$[(\rho_1,\rho_1),\cdot]$ is $Y$-harmonic.
\end{lemma}
The proof follows from Proposition \ref{p:both} and the
fact that $(\rho_1,\rho_1) \in {\mathcal H} \cap {\mathcal H}_1.$

Fixing $\beta \in {\mathbb C}$ and solving \eqref{e:pasalpha} gives us the two conjugate
points corresponding to $\beta$. It is also natural to start the computation
from a fixed $\alpha$ and find its $\beta$ and its conjugate. For this,
one rewrites ${\bm p} =1$, now as a polynomial in $\beta$:
\begin{equation}\label{e:pasbeta}
\left(\mu_1 + \frac{\mu_2}{\alpha}\right)\beta^2  -\beta + \lambda_1 + \lambda_2 \alpha = 0.
\end{equation}
For $\alpha$ fixed, the roots of \eqref{e:pasbeta} are
\begin{equation}\label{e:betas}
\beta_1(\alpha) = \frac{1 - \sqrt{\Delta(\alpha)}}{2\left( \frac{\mu_2}{\alpha} + \mu_1 \right)},~~
\beta_2(\alpha) = \frac{1 + \sqrt{\Delta(\alpha)}}{2\left( \frac{\mu_2}{\alpha} + \mu_1 \right)},
\end{equation}
where
\[
\Delta(\alpha) = 1 - 4\left(\frac{\mu_2}{\alpha} + \mu_1\right)(\lambda_1  + \lambda_2 \alpha), 
\]
and for $z \in {\mathbb C}$, $\sqrt{z}$ is the square root of $z$ satisfying $\Re({\sqrt z}) \ge 0.$

The function $y \mapsto P_y( \tau < \infty)$ takes the value $1$ on $\partial B$; therefore,
of special significance to us is the solution of \eqref{e:pasbeta} with $\alpha=1$. The roots
\eqref{e:betas} for $\alpha=1$ are
\[
\beta_1 = r, \beta_2 = 1.
\]
That $r  \le \rho_1 < 1$ implies
$C(r,1) = (1- r) \neq 0$.
The assumption $\rho_1  \neq \rho_2$ implies
\begin{equation*}
C(r,\bm\alpha(r,1)) = 1-\rho_2/r \neq 0.
\end{equation*}
Therefore, by Proposition \ref{p:conjugate}, the root $\beta_1=r$ above defines the
$Y$-harmonic function
\begin{align*}
h_r &= C(r,\bm\alpha(r,1)) [(r,1),\cdot]  - C(r,1)[(r,\bm\alpha(r,1)),\cdot],\\
 & =(1-\rho_2/r) [(r,1),\cdot]  - (1-r)[(r,r^2/\rho_2),\cdot].
\end{align*}

For this function to be useful in our analysis, 
we need $r^2/\rho_2 < 1$ (see Proposition \ref{p:pbdetermined}),
therefore, we assume:
\begin{equation}\label{as:alphaconj}
\frac{r^2}{\rho_2} < 1.
\end{equation}
Finally, the following scalar multiple of $h_r$ is frequently used
in the calculations, therefore, we will denote it in bold thus:
\begin{equation}\label{d:boldh}
{\bm h}_r = \frac{1}{1-\rho_2/r} h_r = 
 [(r,1),\cdot]  - \frac{1-r}{1-\rho_2/r}[(r,r^2/\rho_2),\cdot];
\end{equation}
the assumption 
$\rho_1 \neq \rho_2$ ensures that the denominator
$1-\rho_2/r$ is nonzero.

\section{Laplace transform of $\tau$}\label{s:laplace}
To bound approximation errors we will have to argue that we can
truncate time without losing much probability. For this, it will
be useful to know that there exists $z > 1$ such that
\begin{equation}\label{e:exponentialbound}
{\mathbb E}_y\left[ z^\tau 1_{\{\tau < \infty \}}\right] < \infty.
\end{equation}
In \cite{sezer2015exit, DSW}, bounds similar to this are obtained
using large deviations arguments, which are based on the ergodicity
of the underlying chain. In \cite{setayeshgar2013efficient}, again a similar bound is
obtained invoking the geometric ergodicity of the underlying process.
The process underlying \eqref{e:exponentialbound} is not stationary. For this reason,
these arguments do not immediately generalize to the analysis of
\eqref{e:exponentialbound}. To prove the existence of $z >1$ such
that \eqref{e:exponentialbound} holds, we will extend the
characteristic surface an additional dimension
 to include a new parameter; 
points on the generalized surface will correspond to discounted 
(in our case we are in fact interested in inflated costs)
expected cost functions of the process $Y$, i.e., points on this
surface will give us functions of the form
\[
{\mathbb E}_y\left[ z^{\tau} g(\tau) 1_{\{\tau < \infty\}} \right].
\]
We will use these functions to find our desired $z$.

\subsection{ $1/z$-level characteristic surfaces and 
$Y$-$z$-harmonic functions}
The development in this subsection is parallel to Section
\ref{s:Yharmonic} with an additional variable $z \in {\mathbb C}$.
A function $h$ on ${\mathbb Z}\times {\mathbb Z}_+$ is said to be 
$Y-z$-harmonic if
\[
z {\mathbb E}_y[ h(Y_1)]= h(y), y \in {\mathbb Z}\times {\mathbb Z}_+.
\]
As before, let ${\bm p}$ denote the characteristic polynomial
of $Y$; 
the set of all solutions of the
equation 
$z{\bm p} = 1$
defines the $1/z$-level characteristic surface
\[
{\mathcal H}^z \doteq \{ (\beta,\alpha) \in {\mathbb C}^2: 
z{\bm p}(\beta,\alpha) =1  \}
\]
Similarly, define
\[
{\mathcal H}^z_1 \doteq \{ (\beta,\alpha) \in {\mathbb C}^2: 
z{\bm p}_1(\beta,\alpha) =1  \},
\]
the $1/z$-level characteristic surface on $\partial_1.$
These surfaces reduce to the ordinary characteristic surfaces 
when $z=1$.

Multiplying both sides of $z{\bm p}=1$ by $\frac{\alpha}{z}$
transforms it to the quadratic (in $\alpha$) equation
\[
\alpha  \left( \lambda_1\frac{1}{\beta}+ \mu_1\beta -\frac{1}{z}\right)
+  \alpha^2 \frac{\lambda_2 }{\beta} +
\mu_2 \beta  = 0,
\]
whose discriminant is
\[
\Delta_z(\beta) \doteq 
\left( \lambda_1 \frac{1}{\beta} + \mu_1 \beta - \frac{1}{z}\right)^2
- 4 \lambda_2 \mu_2.
\]
Let $\bm \alpha$ be the conjugator defined in \eqref{e:conjugator}.
If $(\beta,\alpha_1) \in {\mathcal H}_z$ and $\alpha \neq 0$
then
 $(\beta,\alpha_2,z) \in  {\mathcal H}_z$ 
for $\alpha_2  =\boldsymbol \alpha(\beta,\alpha_1)$; 
if $\Delta_z(\beta) \neq 0$
$(\beta,\alpha_1)$ and $(\beta,\alpha_2)$ will be distinct points on
${\mathcal H}_z$ and we will call them conjugate.

\begin{lemma}\label{l:interior}
$[(\beta,\alpha),\cdot]$ is $Y-z$
 harmonic on ${\mathbb Z}\times {\mathbb Z}_+ - \partial_2$
when $(\beta,\alpha) \in {\mathcal H}_z$. 
In addition $ x\mapsto [(\beta,\alpha), T_n(x) ]$, $x
\in {\mathbb Z}^2_+$, is $X-z$-harmonic 
on ${\mathbb Z}_+^2 - \partial_1 \cup \partial_2.$
\end{lemma}
The proof is parallel to that of Lemma \ref{l:interior0} 
and follows from the definitions.

Define
\[
C_z(\beta,\alpha)  \doteq z  \left( 1- \frac{\beta}{\alpha} \right),
(\beta,\alpha) \in {\mathbb C}^2, \alpha \neq 0.
\]
Parallel to Section \ref{s:Yharmonic},
the above definitions give us the following class of $Y-z$-harmonic functions;
\begin{proposition}\label{p:bothYz}
Suppose $(\beta,\alpha) \in {\mathcal H}^z \cap {\mathcal H}_1^z$.
Then $[(\beta,\alpha),\cdot]$ is $Y-z$-harmonic.
\end{proposition}
\begin{proposition}\label{p:conjYz}
Suppose $(\beta,\alpha_1) \neq (\beta,\alpha_2)$,  are conjugate points
on ${\mathcal H}^z.$ Then
\begin{equation}\label{d:hbeta}
h_{z,\beta} 
\doteq C_z(\beta,\alpha_2) [(\beta,\alpha_1),\cdot]- C_z(\beta,\alpha_1)
[(\beta,\alpha_2),\cdot]
\end{equation}
is $Y-z$-harmonic.
\end{proposition}
The proofs are the same as those of the corresponding Propositions \ref{p:both} and \ref{p:conjugate}
of the previous section.

\subsection{Existence of the Laplace transform of $\tau$}
We next use the $Y-z$ harmonic functions constructed in Propositions
\ref{p:bothYz} and \ref{p:conjYz} to get our existence result.
\begin{proposition}\label{p:laplace}
There exist $z_0 > 1$ and $C_1$ such that
\begin{equation}\label{e:toprovelaplace}
{\mathbb E}_y[ z_0^{\tau} 1_{\{ \tau < \infty \}}] < C_1
\end{equation}
for all $y \in {\mathbb Z} \times {\mathbb Z}_+, y(1) \ge y(2).$
\end{proposition}
\begin{proof}
Let us first prove the following:
if we can find, for some $z_0 > 1$ and $C_1 > 0$,
a $Y$-$z_0$  harmonic function $h$
satisfying $h(y) \ge 1$ on $\partial B$ and $C_1 > h \ge 0$ on $B$
we are done. 
The reason is as follows:
that $h$ is $Y-z_0$-harmonic and the option sampling
theorem imply that  $h(Y_{\tau \wedge n})z_0^{\tau \wedge n}$ is
a martingale.
It follows  that
\begin{align*}
 h(y) &=
{\mathbb E}_y[h(Y_{\tau \wedge n})z_0^{\tau \wedge n}] 
\intertext{
for $ y\in B.$ Decompose the last expectation to $\{ \tau \le n \}$
and $\{\tau > n \}$:}
h(y) &=
{\mathbb E}_y[
h(Y_{\tau })z_0^{\tau}1_{\{\tau \le n \}}] +  
{\mathbb E}_y[
h(Y_{n})z_0^{ n}1_{\{\tau > n \}}]
] 
\intertext{That $h \ge 0$ on $B$ implies}
h(y) &\ge
{\mathbb E}_y[
h(Y_{\tau \wedge n })z_0^{\tau}1_{\{\tau \le n \}}]
\intertext{Now $\lim_{n\rightarrow \infty } 
h(Y_{\tau}) z_0^{\tau} 1_{\{\tau \le n\}} = 
h(Y_{\tau}) z_0^{\tau} 1_{\{\tau < \infty\}}.$ This and Fatou's lemma
imply}
h(y) &\ge
{\mathbb E}_y[
h(Y_{\tau})z_0^{\tau}1_{\{\tau < \infty \}}]
\end{align*}
Finally, $h \ge 1 $ on $\partial B$ and $h \le C_1$ give
\eqref{e:toprovelaplace}.

To get our desired $h$ we start from the points
$(r,1)$ and  $(\rho_1,\rho_1)$
on ${\mathcal H}$.
The first point gives us the root $(1,r)$ of the equation
\[
z{\bm p}(\beta,1) = 1.
\]
That 
\[
\frac{\partial  z {\bm p }(\beta,\alpha)}{\partial \beta} |_{(1,r,1)}=
(\lambda_1 + \lambda_2) \frac{r-1}{r^2} \neq 0
\]
and the implicit function theorem give us a differentiable
function $\bm \beta_1$ on an open interval $I_1$
around $z=1$ that satisfies
\[
z {\bm p}(\bm\beta_1(z),1) = 1, z \in I_1, {\bm \beta}_1(1) = r.
\]
The conjugate of $({\bm \beta}_1,1)$ on ${\mathcal H}^z$ is
$({\bm \beta}_1,{\bm\alpha}({\bm \beta}_1,1))$ 
(whenever possible, 
we will omit the $z$ variable and simply write $\bm \beta_1$;
similarly, we will write $\bm \alpha$ for ${\bm \alpha}({\bm \beta}_1,1))$.
These points give us
the $Y-z$-harmonic function
\[
{\bm h}_z = [(\bm \beta_1, 1),\cdot] - 
\frac{1-\bm \beta_1}{1- \rho_2/\bm \beta_1} [(\bm \beta_1,\bm \alpha),\cdot],
\]
where we used 
$C(\bm \beta_1,1)/C(\bm \beta_1,\bm \alpha) = \frac{1-\bm\beta_1}{1-\rho_2/\bm\beta_1}.$
That $\bm \alpha(\bm \beta_1(1),1) = 0 <  r^2/\rho_2< 1$ 
(Assumption \ref{as:alphaconj})implies that 
$0 < \bm \alpha(\bm \beta_1(z),1) <1$ if we choose 
$z >1$ close enough to $1$. $\bm h_z$ will almost serve as our $h$,
except that it does take negative values on a small section of $B$.
To get a positive function we will add to $\bm h_z$ a constant
multiple of the $Y-z$-harmonic function defined by a
point on ${\mathcal H}^z \cap {\mathcal H}_2^z$ that is the continuation
of $(\rho_1,\rho_1)$ on ${\mathcal H}.$ This point is 
$(\bm \beta_2(z),\bm \beta_2(z))$ where $\bm \beta_2(z)$ is the root of the
the equation
\[
\left( \frac{\lambda_1}{\beta} + \mu_1 \beta + \lambda_2 + \mu_2 \right)
=  \frac{1}{z};
\]
satisfying $\bm \beta_2(1) = \rho_1.$ The implicit function theorem
(or direct calculation) shows that $\bm \beta_2$ is smooth in 
an open interval $I_2$ containing $1.$ Now 
$(\bm \beta_2, \bm\beta_2) \in {\mathcal H}^z \cap {\mathcal H}_2^z$
and Proposition \ref{p:bothYz} 
imply that $[(\bm \beta_2,\bm \beta_2), \cdot] \ge 0$
is a $Y-z$ harmonic function. Now define
\[
h' \doteq \bm h_z + C_0 [(\bm \beta_2, \bm \beta_2), \cdot].
\]
By its definition $h'$ is $Y-z$ harmonic.
We would like to choose $C_0$ large enough so that $h'$ is bounded
below by $1$ on $\partial B$ and is nonnegative on $B.$
By our assumption \eqref{as:orderrho},
$\bm \beta_1(1) = r < \bm \beta_2(1) = \rho_1$;
therefore, for $z >1$ close enough to 
$1$, we will still have $\bm \beta_1(z) < \bm \beta_2(z)$; let us
assume that $I_1$ and $I_2$ are tight enough that this holds.
By definition,
\[
h'(y) = \bm \beta_1^{y(1)-y(2)} 
\left( 1 - \frac{1-\bm \beta_1}{1-\rho_2/\bm \beta_1} \bm \alpha^{y(2)} \right) + C_0
\bm \beta_2^{y(1)-y(2)} \bm \beta_2^{y(2)}.
\]
$\bm \beta_2 > \bm \beta_1$ implies that $h'$ 
takes its most negative value for $y(1) = y(2)$,
i.e., on $\partial B$ and if we can choose $C_0 > 0$ so that $h$ is nonnegative on $\partial B$,
it will be so on all of $B$. On $\partial B$, $h'$ reduces to
\[
1 - \frac{1-\bm \beta_1}{1-\rho_2/\bm \beta_1} \bm \alpha^{y(2)} + C_0
\bm \beta_2^{y(2)}.
\]
If $\bm \alpha \le \bm \beta_2$, then setting 
$C_0 = \frac{ 1- \bm \beta_1 }{1-\rho_2/\bm \beta_1}$
would imply $h' \ge 1$ on $\partial B$; $\bm \alpha, \bm \beta_1, \bm \beta_2  \in (0,1)$ imply
\[
h' \le 1 + C_0;
\]
then, $h'$ can serve as our desired $Y-z$ harmonic function $h$ with
$C_1 = 1 + C_0.$

Now let us consider the case
$ \bm \alpha >  {\bm \beta}_2$ :ordinary calculus implies that if we choose
$C_0$ large enough we can make the minimum $m_0 < 0$ over $y(2) > 0$ of 
\[
- \frac{1-\bm \beta_1}{1-\rho_2/\bm \beta_1} \bm \alpha^{y(2)} + C_0
\bm \beta_2^{y(2)}
\]
arbitrarily close to $0$; then choosing $h = \frac{1}{1+m_0} h'$ gives
us a $Y-z$ harmonic function that satisfies $h \ge 1$ on $\partial B$, 
$h \ge 0$ on $B$ and $h \le C_1$ on $B$ where $C_1 = \frac{1}{1+m_0}( 1 + C_0).$
\end{proof}

We now use \eqref{e:toprovelaplace}
to derive an upper bound on the probability that $\tau$ is finite
but too large:
\begin{proposition}
For any $ \delta >0$, there exists $  C_2 > 0$ such that
\begin{equation}\label{e:boundontau}
P_{y}( n  C_2 < \tau < \infty) \le e^{-\delta n}
\end{equation}
for any $ y\in {\mathbb Z} \times {\mathbb Z}_+$, $y(1) > y(2)$
and $n > 1$.
\end{proposition}
\begin{proof}
Let $z_0 > 1$ and $C_1$ be as in \eqref{e:toprovelaplace}.
For any $A > 0$, Chebyshev's inequality gives
\begin{align*}
P_y( n  C_2 < \tau <\infty) &= P_y( z_0^{n  C_2} < z_0^\tau <\infty)\\
&\le {\mathbb E}_y[z_0^\tau < \infty] z_0^{-n  C_2}\\
&\le  e^{-n ( C_2 \log(z_0) - \log(C_1)/n)}.
\end{align*}
Choosing $ C_2 = (\delta +\log(C_1)) /\log(z_0) $ gives \eqref{e:boundontau}.
\end{proof}

\section{LD limit for $P_x(\tau_n < \tau_0)$}\label{s:LD1}
Define
\[
V(x) \doteq \log \rho_1 (x(1) -1) \wedge \log(r) (x(1) + x(2) - 1) \wedge \log\rho_2(x(2) - 1).
\]
Assumption \eqref{as:orderrho} implies
\[
-\log(\rho_2) (1-x(2)) \ge  -\log(r) (1- (x(1) + x(2)) )
\]
and therefore
\begin{equation}\label{e:defV}
V(x) = \log(r) (x(1) + x(2) - 1) \wedge \log\rho_1(x(1) - 1).
\end{equation}
The level curves of $V$ for 
for $\lambda_1 =0.2$, $\lambda_2 = 0.1$, $\mu_1 = 0.3$, $\mu_2 = 0.4$
are shown in Figure \ref{f:ldrate}.

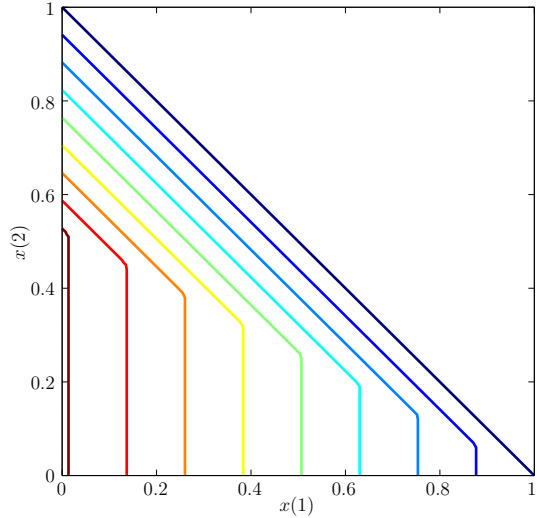
\begin{figure}[h]
\centering
\hspace{-0.5cm}
\scalebox{0.5}{
\input{ldrate}
}
\caption{Level curves of $V$
for $\lambda_1 =0.2$, $\lambda_2 = 0.1$, $\mu_1 = 0.3$, $\mu_2 = 0.4$}
\label{f:ldrate}
\end{figure}

The goal of this section is to prove 
\begin{theorem}
$V$ is the LD limit of $P_x(\tau_n < \tau_0)$, i.e.,
\begin{equation}\label{e:LDlimitp}
\lim -\frac{1}{n} \log P_{\lfloor nx \rfloor}( \tau_n < \tau_0) = V(x).
\end{equation}
for $x(1) + x(2) < 1$, $ x \in {\mathbb R}_+^2.$
\end{theorem}
\begin{proof}
Propositions \ref{p:LDlow} and \ref{p:LDupper} state
\begin{equation}\label{e:LDlowerbound}
\liminf -\frac{1}{n} \log P_{\lfloor nx \rfloor}( \tau_n < \tau_0) \ge
V(x).
\end{equation}
and
\begin{equation}\label{e:LDupperbound}
\limsup -\frac{1}{n} \log P_{\lfloor nx \rfloor}( \tau_n < \tau_0) \le
V(x).
\end{equation}
These imply \eqref{e:LDlimitp}.
\end{proof}
Next two subsections prove \eqref{e:LDlowerbound}
and \eqref{e:LDupperbound}.
To prove the first, we will proceed parallel to 
\cite{thesis,DSW,sezer2015exit} 
and construct a sequence of supermartingales $M^n$
starting from a subsolution of a limit Hamilton Jacobi Bellman (HJB)
equation associated with the problem.
To prove the bound \eqref{e:LDupperbound} we will directly construct a 
sequence of subharmonic functions of the process $X$. 
\subsection{LD lowerbound for $P_x(\tau_n < \tau_0)$}\label{ss:ldlow}

For $a \subset \{0,1,2\}$ define the Hamiltonian function
\[
H_a(q) \doteq -\log \left(
\sum_{v \in a^c} p(v)e^{-\langle q,v \rangle} +  \sum_{v \in {\mathcal V}, v(i) \ge 0, i \in a } p(v)
e^{-\langle q,v \rangle} + \sum_{v \in {\mathcal V}, v(i) < 0, i \in a } p(v)\right).
\]
We will denote $H_\emptyset$ by $H$. $H_a$ is convex in $q$.
For $ x \in {\mathbb R}_+^2$, define
\[
\bm b(x) \doteq \{i: x(i) = 0 \}.
\]
Following \cite{dupell2} one can represent $V$ as the value function
of a continuous time deterministic control problem; the 
HJB equation associated with this control problem
is
\begin{equation}\label{e:HJB}
H_{\bm b(x)}(DV(x)) =0;
\end{equation}
a function $W \in C^1$ is said to be a classical subsolution
of \eqref{e:HJB} if
\begin{equation}\label{e:HJBsub}
H_{\bm b(x)}(DV(x)) \ge 0;
\end{equation}
supersolutions are defined by replacing $\ge$ in \eqref{e:HJBsub}
 with $\le.$

To prove \eqref{e:LDlowerbound}
will proceed parallel to \cite[Section 7]{sezer2015exit}:
find an upperbound on $P_x(\tau_n < \tau_0)$ by constructing 
a supermartingale 
associated with the process $X$.
To construct our supermartingale we will proceed parallel to 
\cite{thesis,DSW} and use a subsolution of \eqref{e:HJB}, i.e.,
a solution of \eqref{e:HJBsub}.

Define
\begin{equation}\label{e:roots}
r_0\doteq (0,0),
~~r_1 \doteq \log(\rho_1)(1,0),
~~ r_2 \doteq \log(r)(0,1),
~~r_3 \doteq \log(r) (1,1)
\end{equation}
and
\begin{align*}
\tilde{V}_0(x,\epsilon) &\doteq  -\log(\rho_1) -3\epsilon, ~~~
\tilde{V}_1(x,\epsilon) \doteq -\log(\rho_1) + \langle r_1, x \rangle - 2\epsilon,\\
\tilde{V}_2(x,\epsilon) &\doteq -\log(r)+ \langle r_2, x \rangle -\epsilon,~~\\
\tilde{V}_3(x,\epsilon) &\doteq -\log(r)+ \langle r_3, x \rangle
\end{align*}
and
\begin{equation}\label{d:Vtild1}
\tilde{V}(x,\epsilon) \doteq \bigwedge_{i=0}^3 \tilde{V}_i(x,\epsilon).
\end{equation}

A direct  calculation gives
\begin{lemma}\label{l:roots}
The gradients defined in \eqref{e:roots} satisfy
\begin{align*}
H(r_0) &= H_1(r_0) = H_2(r_0) = 0,~~~ H(r_1) = H_2(r_1) = 0, \\
H(r_2) &>0,  H_1(r_2) > 0, ~~~~~~ H(r_3) =0
\end{align*}
\end{lemma}
The $0$-level curve of the Hamiltonians and the
gradients $r_i$ are shown in Figure \ref{f:charsurf}
\begin{figure}[h]
\centering
\hspace{-1cm}
\begin{subfigure}{0.48\textwidth}
\scalebox{0.45}{
\input{charsurf}}
\caption{
The $0$-level curves of $H$ and $H_1$ (dashed line)
for $\lambda_1 =0.2$, $\lambda_2 = 0.1$, $\mu_1 = 0.3$, $\mu_2 = 0.4$
and the gradients $r_i$, $i=0,1,2,3$.
}
\label{f:charsurf}
\end{subfigure}
\hspace{0.5cm}
\begin{subfigure}{0.48\textwidth}
\scalebox{1}{
\input{regionstex}
}
\caption{ Regions $R_i$}
\label{f:regions}
\end{subfigure}
\end{figure}

Define
\begin{equation}\label{d:C4}
C_3 \doteq -\frac{1}{\log(r)}, C_4 \doteq -\log(\rho_1)C_3 = \frac{\log(\rho_1)}{\log(r)}.
\end{equation}
$ \log(r) < \log(\rho_1) < 0$ implies 
\begin{equation}\label{e:boundbeta}
1 > C_4 > 0.
\end{equation}
The functions $\tilde{V}_i$, $i=1,2,3$ meet at
\begin{equation}\label{d:xstar}
x^* = 
\left( C_3\epsilon, 1 - C_4  +C_3(1+C_4)\epsilon \right) 
\end{equation}
i.e., 
\begin{equation}\label{e:xstar}
\tilde{V}_1(x^*) = \tilde{V}_2(x^*) = \tilde{V}_3(x^*) = -\log(\rho_1) -
\left(2 + C_4 \right)\epsilon.
\end{equation}
We assume that $\epsilon> 0$ is small enough so that  $x^*$
satisfies
\[
x^*(1), x^*(2) > 0, x^*(1) + x^*(2) < 1.
\]
$\tilde{V}(\cdot,\epsilon)$ equals $\tilde{V}_i(\cdot,\epsilon)$ in
the region
\[
R_i \doteq \{x \in {\mathbb R}^2: \tilde{V}(x,\epsilon) = \tilde{V}_i(x,\epsilon)\},
\]
these regions are shown in Figure \ref{f:regions}.

As in \cite{thesis,DSW}, 
we will mollify $\tilde{V}(x,\epsilon)$ with
\[
\eta(x) \doteq 1_{|x| \le 1 } (|x|^2-1)^2, x \in {\mathbb R}^2,
C_5 \doteq \int_{\mathbb R^2} \eta(x) dx,
\eta_\delta(x) \doteq \frac{1}{\delta^2 C_5} \eta(x/\delta),  \delta > 0.
\]
to get our smooth subsolution of \eqref{e:HJB}:
\begin{equation}\label{d:defVsmooth}
V(x,\epsilon) \doteq \int_{{\mathbb R}^2} \tilde{V}(x+y,\epsilon) 
\eta_{{0.5}C_3 \epsilon}(y)dy,
\end{equation}
\begin{lemma}\label{l:subsol}
The function
$V(x,\epsilon)$ of \eqref{d:defVsmooth} satisfies \eqref{e:HJBsub} and
\begin{equation}\label{e:boundsecder}
\left| \frac{\partial^2 V(\cdot,\epsilon)}{\partial x_i \partial x_j}
\right| \le C_6/\epsilon
\end{equation}
where $C_6$ is independent of $x.$
Furthermore,
\begin{equation}\label{e:Vboundary}
V(x,\epsilon) \le \epsilon \text{ for } x(1) + x(2) = 1, x \in {\mathbb R}_+^2.
\end{equation}
\end{lemma}
\begin{proof}
The proof is parallel to that of \cite[Lemma 2.3.2]{thesis}.
$\tilde{V}$ is the minimum of four affine functions and hence is Lipschitz continuous
and has a bounded (piecewise constant) gradient almost everywhere. 
This implies
\begin{align}\label{e:DV}
DV(x,\varepsilon) &= 
\int_{{\mathbb R}^2} D\tilde{V}(x+y,\epsilon) \eta_{0.5 C_3\epsilon}(y)dy
\notag\\
&=\sum_{i=0}^3 w_i(x) D\tilde{V}_i,
=\sum_{i=1}^3 w_i(x) r_i
\end{align}
where
\[
w_i(x) = \int_{{\mathbb R}^2} \eta_{0.5 C_3 \epsilon}(x) 1_{R_i}(x) dx.
\]
This shows that $V(\cdot, \varepsilon) \in C^1.$
To show 
\begin{equation}\label{e:toshow}
H_{\bm b(x)}(DV(\cdot,\epsilon)) \ge 0,
\end{equation} one considers
$x\in {\mathbb R}_+^{2o} \doteq \{x \in {\mathbb R}_+^2, x(1),x(2) > 0\}$,
$x \in \partial_1$ and $x \in \partial_2$ separately. 
We will provide the details only for the first two.
For $x \in {\mathbb R}_+^{2o}$,
\[
H_{\bm b(x)}(DV(\cdot,\epsilon)) = H(DV(\cdot,\epsilon)).
\]
By Lemma \ref{l:roots} we know that all $r_i$ satisfy $H(r_i) \ge 0.$
That $H$ is a convex function and Jensen's inequality imply that
the $DV(x,\epsilon) = \sum_{i=0}^3 w_i(x) r_i$ satisfies
$H(DV(x,\epsilon)) \ge 0$; this proves \eqref{e:toshow} for 
$x \in {\mathbb R}_+^{2o}.$

We know by \eqref{e:boundbeta} that $C_4 \in (0,1)$; therefore,
by \eqref{e:xstar}
\begin{align*}
\tilde{V}_0(\cdot,\epsilon) = -\log(\rho_1) - 3\epsilon 
&<-\log(\rho_1) - (2 +C_4) \epsilon\\
&=\tilde{V}_1(x^*,\epsilon) = \tilde{V}_2(x^*,\epsilon) = \tilde{V}_3(x^*,\epsilon).
\\
\end{align*}
This implies that the region $R_0$ intersects all of the
$R_1$, $R_2$ and $R_3$ and in particular that the strip
$\{x \in {\mathbb R}_+^2: x(1)  < C_3 \epsilon \}$ lies in 
$R_0 \cup R_2$.
Then, for $x \in \partial_1$, the ball $B(x,C_3\epsilon/2)$
lies completely in $R_0 \cup R_2$, which implies
\[
DV(x,\epsilon) = w_0(x) r_0 + w_1(x) r_2, w_0(x) + w_1(x) = 1.
\]
By Lemma \ref{l:roots} we know that $H_1(r_0) = 0$ and $H_1(r_2) \ge 0.$
These and the convexity of $H_i$  imply \eqref{e:toshow} for $x \in \partial_1.$

The bound \eqref{e:boundsecder} follows from the Lipschitz continuity
of $\tilde{V}(\cdot,\epsilon)$, differentiation under the
integral sign in \eqref{d:defVsmooth} and bounds on the first derivative of
$\eta$.
Finally, \eqref{e:Vboundary} follows from 
\[
\tilde{V}_3(x,\epsilon) \le \frac{\epsilon}{\sqrt{2}},
\text{ for any }x \in B(y, \epsilon C_3 /2), ~y \text{ such that }
y(1) + y(2) = 1, y\in {\mathbb R}_+^2,
\]

\end{proof}

To get our upperbound on the probability $P_x( \tau_n < \tau_0)$ we define
\[
M^{(n,\epsilon)}_k \doteq e^{-n V(X_k/n,\epsilon)-\frac{C_6k}{n\epsilon}},
\]
where $C_6$ is as in \eqref{e:boundsecder}.
\begin{lemma}
$M^{(n,\epsilon)}$ is a supermartingale.
\end{lemma}
\begin{proof}
The Markov property of $X$ implies that it suffices to show
\begin{align}
{\mathbb E}_x\left[ M^{(n,\epsilon)}_1 \right] &\le e^{-nV(x/n,\epsilon)}\notag\\
{\mathbb E}_x\left[ M^{(n,\epsilon)}_1 e^{-nV(x/n,\epsilon)}  \right] &\le 1\notag\\
-\log\left({\mathbb E}_x\left[ M^{(n,\epsilon)}_1 e^{-nV(x/n,\epsilon)}  \right]
\right) &\ge 0.\label{e:toprove}
\end{align}
The expression on the left equals
\begin{equation}\label{e:explicitexp}
-\log \left(
\sum_{v: v(i) \ge 0, i \in {\bm b}(x)} 
e^{-n \left(V( (x+v)/n,\epsilon)- V(x/n,\epsilon)\right)} p(v) + 
\sum_{v: v(i) =-1, i \in {\bm b}(x)} p(v) \right) + \frac{C_6}{n\epsilon}.
\end{equation}
A Taylor expansion and the bound \eqref{e:boundsecder} imply
\[
\left|
\left(V( (x+v)/n,\epsilon)- V(x/n,\epsilon)\right) -
\langle DV(x), v/n \rangle \right| \le \frac{C_6}{n^2\epsilon}
\]
Then, the expression in \eqref{e:explicitexp} is bounded below by
\[
-\log \left(
\sum_{v: v(i) \ge 0, i \in {\bm b}(x)} 
e^{-\langle DV(x), v \rangle } p(v) + 
\sum_{v: v(i) =-1, i \in {\bm b}(x)} p(v) \right) -\frac{C_6}{n\epsilon}+ 
\frac{C_6}{n\epsilon}.
\]
The $\log$ term above equals $H_{\bm b(x)}(DV(x,\epsilon))$, which
by Lemma \ref{l:subsol} is nonnegative. This proves \eqref{e:toprove}.

\end{proof}
\begin{proposition}\label{p:LDlow}
Let $x \in {\mathbb R}_+^2$  with
$x(1) + x(2) < 1$, $x_n = \lfloor n x \rfloor$ and let $V$ be as in \eqref{e:defV}. Then for any $\varepsilon > 0$
there exists an integer $N$ such that for $n > N$
\begin{equation}\label{e:upperboundP}
P_{x_n}( \tau_n < \tau_0) \le e^{-n(V(x)-\varepsilon)}.
\end{equation}
In particular,
\begin{equation}\label{e:LDlimitprob}
\liminf- \frac{1}{n} \log P_{\lfloor nx \rfloor}( \tau_n < \tau_0) \ge V(x).
\end{equation}
\end{proposition}

The proof is parallel to that of \cite[Proposition 4.3]{tandemsezer}. 
\begin{proof}
The inequality \eqref{e:LDlimitprob} follows from \eqref{e:upperboundP} upon taking limits. The rest of the
proof focuses on \eqref{e:upperboundP}.
Let $\epsilon_n > 0$ be a sequence satisfying $\epsilon_n\rightarrow 0$
and $\epsilon_n n \rightarrow \infty.$ Let $\tau_{0,n} = \tau_n \wedge
\tau_0.$ The optional sampling theorem (\cite[Theorem 5.7.6]{durrett2010probability}) applied to the supermartingale $M_k = M_k^{(n,\epsilon_n)}$ 
at time $\tau_{0,n}$ gives
\[
{\mathbb E}_{x_n} \left[
M_{\tau_{0,n}}\right] \le M_0 =  e^{-nV(x_n/n,\epsilon_n)}.
\]
Restricting the expectation on the left to $\{\tau_n < \tau_0\}$ makes it
smaller:
\begin{align}
{\mathbb E}_{x_n} \left[1_{\{ \tau_n < \tau_0\}} M_{\tau_{n}}\right] &\le   e^{-nV(x_n/n,\epsilon_n)} \notag
\intertext{Expanding $M_{\tau_n}$ using its definition gives}
{\mathbb E}_{x_n} \left[1_{\{ \tau_n < \tau_0\}} e^{-n V(X_{\tau_n}/n,\epsilon_n)} e^{\frac{-C_6\tau_n}{n\epsilon_n}}\right] &\le   e^{-nV(x_n/n,\epsilon_n)}\notag
\intertext{$X_{\tau_n} \in \partial A_n$ and the bound \eqref{e:Vboundary} reduce the last display to}\label{e:almostthere}
{\mathbb E}_{x_n} \left[1_{\{ \tau_n < \tau_0\}} e^{\frac{-C_6\tau_n}{n\epsilon_n}}\right] &\le   e^{-nV(x_n/n,\epsilon_n)}
e^{n\epsilon_n}
\end{align}
By the definitions involved we have
\[
\lim_{n\rightarrow \infty} V(x_n,\epsilon_n) = V(x).
\]
This, $n\epsilon_n \rightarrow 0$ 
and taking the $\liminf -\frac{1}{n}\log$ of both sides in
\eqref{e:almostthere} gives
\begin{equation}\label{e:liminfofE}
\liminf_{n\rightarrow \infty} -\frac{1}{n}\log
{\mathbb E}_{x_n} \left[1_{\{ \tau_n < \tau_0\}} e^{\frac{-C_6\tau_n}{n\epsilon_n}}\right]  \ge V(x).
\end{equation}
Now suppose that \eqref{e:upperboundP} doesn't hold, i.e., there exists
$\varepsilon > 0$ and a sequence $n_k$ such that
\begin{equation}\label{e:contradict}
P_{x_{n_k}}( \tau_n < \tau_0) > e^{-n_k (V(x) -\varepsilon)}
\end{equation}
for all $k$; we pass to this subsequence and omit the subscript $k$.
\cite[Theorem A.1.1]{thesis} implies that there is a $ C_2 > 0$ such
that
\begin{equation}\label{e:toolongtime}
P_{x_n}(\tau_{0,n} > n C_2) \le e^{-n (V(x) + 1)}
\end{equation}
for $n$ large.
Then
\begin{align*}
{\mathbb E}_{x_n}
\left[ 1_{\{\tau_n < \tau_0\}} e^{-\frac{C_6 \tau_n}{n\epsilon_n}} \right]
&\ge 
{\mathbb E}_{x_n}
\left[ 1_{\{\tau_n < \tau_0\}} e^{-\frac{C_6 \tau_n}{n\epsilon_n}}
1_{\{\tau_{0,n} \le n C_2\}} \right]\\
&\ge e^{\frac{-C_6  C_2}{n\epsilon n} n}
{\mathbb E}_{x_n}
\left[ 1_{\{\tau_n < \tau_0\}} 
1_{\{\tau_{0,n} \le n C_2\}} \right]\\
&\ge e^{\frac{-C_6 C_2}{n\epsilon_n}n}\left( P_{x_n}(\tau_n < \tau_0)
-P_{x_n}(\tau_{0,n} > n  C_2) \right)\\
&\ge e^{\frac{-C_6 C_2}{n\epsilon_n}n}\left(e^{-n(V(x)-\varepsilon)}
-e^{-(V(x)+1)n} \right).
\end{align*}
Now taking $\limsup -\frac{1}{n}\log$ of both sides gives
\[
\limsup_{n\rightarrow \infty} -\frac{1}{n}\log
{\mathbb E}_{x_n} \left[1_{\{ \tau_n < \tau_0\}} e^{\frac{-C_6\tau_n}{n\epsilon_n}}\right]  \le V(x)-\varepsilon.
\]
This contradicts \eqref{e:liminfofE}. Therefore, the assumption \eqref{e:contradict} is false and
there does exist $N > 0$ such that \eqref{e:upperboundP} holds for $n > N$.
This finishes the proof of this proposition.
\end{proof}
\subsection{LD upperbound for $P_x(\tau_n < \tau_0)$}
The LD upperbound corresponds (because of the $-\log$ transform) to a lowerbound on
the probability $P_x(\tau_n < \tau_0).$ To get a lower bound on this probability,
it suffices to have a submartingale of $X$ with the right values when $X$ hits
$\partial A_n \cup \{0\}$.  
As opposed to the analysis of the previous section
(where we constructed a supermartingale from a subsolution to a limit HJB equation), 
one can directly construct a subharmonic function of $X$ to get the desired submartingale.
The next proposition gives this explicit subharmonic function. 
In its proof the following fact will be useful: if $g_1$ and $g_2$ are
subharmonic functions of $X$ at a point $x$, then so is
$g_1 \vee g_2$, this follows from the definitions involved.
\begin{proposition}\label{p:subharmonicpn}
\[
f_n(x) \doteq  {\rho_1}^{n-x(1)} \vee  r^{(n-x(1))-x(2)} \vee \rho_1^{n-1}
\]
is a subharmonic function of $X$ on $A_n - \partial A_n$
\end{proposition}
\begin{proof}
We note
\[
\rho_1^{n-x(1)} = \rho_1^{(n-x(1)) - x(2)} \rho_1^{x(2)} 
= [( \rho_1,\rho_1), T_n(x)].
\]
Furthermore, $(\rho_1,\rho_1) \in {\mathcal H}$.
It follows from these and Lemma \ref{l:interior} that
$x \mapsto \rho_1^{n-x(1)}$ is $X$-harmonic for 
$x \in {\mathbb Z}_+^{2o} \doteq {\mathbb Z}_+^2 -\{\partial_1 \cup \partial_2\}.$
A parallel argument proves the same for 
$x \mapsto r^{(n-x(1))-x(2)}.$ The constant function
$x \mapsto \rho_1^{n-1}$ is trivially $X$-harmonic for
all $x \in {\mathbb Z}_+^2.$ It follows that their maximum,
$f_n$ is subharmonic on  ${\mathbb Z}_+^{2o}.$

It remains to prove that $f_n$ is subharmonic on $\partial_1$ and
$\partial_2.$ $f_n(x) = r^{(n-x(1))-x(2)} \vee \rho_1^{n-1}$ for
$x \in \partial_1 \cup \{x \in {\mathbb Z}_+^2, x(1) = 1 \}.$
Both  $x \mapsto r^{(n-x(1))-x(2)}$ and 
$x \mapsto \rho_1^{n-1}$  are $X$-harmonic on $\partial_1.$
It follows from these that $f_n$ is subharmonic on $\partial_1.$

For  $x \in \partial_2 \cap \{x \in {\mathbb Z}_2^+: x(1) < n\} $
we have $f_n(x) = \rho_1^{n-x(1)} \vee \rho_1^{n-1}.$ By Lemma \ref{l:rho1}
and by fact that $I_k = {\mathcal I} J_k$ we know
$x\mapsto \rho_1^{n-x(1)}$ is harmonic on $\partial_2$;
the same trivially holds for $x \mapsto \rho_1^{n-1}$; therefore,
$x\mapsto \rho_1^{n-x(1)} \vee \rho_1^{n-1}$ is subharmonic on 
$\partial_2 \cap \{x \in {\mathbb Z}_2^+: x(1) < n\} $
Furthermore, by definition $f_n(x) \ge 
\rho_1^{n-x(1)} \vee \rho_1^{n-1}$.
These imply that
$f_n$ is subharmonic on
 $\partial_2 \cap \{x \in {\mathbb Z}_2^+: x(1) < n \}.$

The last three paragraphs together imply the statement of the proposition.
\end{proof}
\begin{proposition}\label{p:LDupper}
\begin{equation}\label{e:discreteupperbound}
P_x(\tau_n < \tau_0)  \ge f_n(x)  -f_n(0)
\end{equation}
and in particular
\begin{equation}\label{e:LDupperboundproof}
\limsup -\frac{1}{n} \log P_{\lfloor nx \rfloor}( \tau_n < \tau_0) \le
V(x),
\end{equation}
for $x \in {\mathbb R}_+^2$, $x(1) +x(2) < 1$, $x(1) > 0.$
\end{proposition}
\begin{proof}
By Proposition \ref{p:subharmonicpn}, we know that $f_n$ is a subharmonic
function of $X$. It follows that $h(X_n)$ is a submartingale. This and
the optional sampling theorem imply:
\begin{align*}
f_n(x) &\le {\mathbb E}_x[ f_n(\tau_n \wedge \tau_0) ]\\
f_n(x) &= P_x(\tau_n < \tau_0) (1-f_n(0)) + f_n(0)\\
 &\le  P_x(\tau_n < \tau_0) + f_n(0),
\end{align*}
where we have used $f_n(x) = 1$ for $ x\in \partial A_n$; this gives
\eqref{e:discreteupperbound}. Taking $-\frac{1}{n}\log$ of both sides
and applying $\limsup$ gives \eqref{e:LDupperboundproof}.
\end{proof}

\section{LD limit of $P_x( \sigma_1 < \tau_n < \tau_0)$}\label{s:LD2}
To implement the argument given in the introduction we need an LD
lowerbound for the probability
\begin{equation}\label{e:boundtwosteps}
P_x( \sigma_1 < \tau_n < \tau_0 ).
\end{equation}
We will obtain the desired bound through a subsolution of the
limit HJB
equation associated with  $X$.
This is parallel to the construction given in 
\cite[Proposition 4.3]{tandemsezer} and the argument of Section \ref{ss:ldlow}.
The main difference is in the construction of the subsolution.
Bounding \eqref{e:boundtwosteps} requires a subsolution
consisting of two pieces, one piece for before $\sigma_1$ and one for after.
For the first piece we need the following additional root of the limit Hamiltonian:
\begin{equation}\label{d:r4}
r_4 \doteq \left( \log(\rho_1/r),\log(r)\right).
\end{equation}
Now define
\begin{align}
\tilde{V}_4(x,\epsilon) &\doteq -\log(r) + \langle r_4, x\rangle,\notag 
\tilde{V}(0,x,\epsilon) \doteq  \bigwedge_{i \in \{0,2,4\}} \tilde{V}_i(x,\epsilon)\notag \\
\tilde{V}(1,x,\epsilon) &\doteq  \tilde{V}(x,\epsilon)= \bigwedge_{i=0}^3 \tilde{V}_i(x,\epsilon)
\label{d:defVsigma1}
\end{align}
and
\begin{align*}
V_\sigma(0,x) &\doteq \tilde{V}(0,x,0) = (-\log(\rho_1)) \wedge (-\log(r) + \langle r_4, x \rangle )\\
V_\sigma(1,x) &\doteq \tilde{V}(1,x,0) = (-\log(\rho_1)+ \langle r_1,x \rangle) 
\wedge (-\log(r) + \langle r_3, x \rangle )
\end{align*}
where the 
vectors $r_i$ are as in \eqref{e:roots}.
Now define the smoothed subsolution:
\begin{equation}\label{d:defVsmooth2s}
V(x,\epsilon,i) \doteq \int_{{\mathbb R}^2} \tilde{V}(x+y,\epsilon,i) 
\eta_{{0.5}C_3 \epsilon}(y)dy, i=0,1.
\end{equation}
The function $\tilde{V}(0,\cdot,\cdot)$ is obtained from $\tilde{V}(1,\cdot,\cdot)$ by striking
out $\tilde{V}_1$ from the minimum and replacing $\tilde{V}_3$ with $\tilde{V}_4.$ In particular,
the components $\tilde{V}_0$ and $\tilde{V}_2$ are common to both 
$\tilde{V}(1,\cdot,\cdot)$ and $\tilde{V}(0,\cdot,\cdot)$; 
this ensures that these functions
overlap around an open region along $\partial_1$, which implies in particular that
\begin{equation}\label{e:equalityalongp1}
V(1,x,\epsilon) = V(0,x,\epsilon)
\end{equation}
for $x \in \partial_1.$
\begin{remark}{\em
The condition \eqref{e:equalityalongp1} allows one to think of
$V(\cdot,\cdot,\cdot)$ as a subsolution of the HJB equation on
a manifold; the manifold consists of two copies of ${\mathbb R}_+^2$,
glued to each other along $\{x \in {\mathbb R}_+^2, x(1) = 0 \}.$
}
\end{remark}
We use $V(\cdot,\cdot,\cdot)$ to construct the supermartingale
\[
M^{(n,\epsilon,\sigma)}_k \doteq e^{-n V(X_k/n,\epsilon,1_{\{k < \sigma_1\}})-\frac{C_6 k}{n\epsilon}},
\]
where $C_6/\epsilon$ is an upperbound on the second derivative of $V(\cdot,\cdot,\cdot)$, which can be obtained by an argument
parallel to the one used
in the proof of \eqref{e:boundsecder} of Lemma \ref{l:subsol}.
The main difference from subsection \ref{ss:ldlow} is that the smooth subsolution has an additional
parameter $i$ to keep track of whether $X$ has touched $\partial_1$; this appears as the
$1_{\{k < \sigma_1\}}$ term in the definition of the supermartingale $M^{(n,\epsilon,\sigma)}$.
A three stage version of this argument appears in \cite[Proposition 4.3]{tandemsezer} to bound 
another related probability arising from the analysis of the two dimensional tandem random walk.
The main result of this section is the following:
\begin{proposition}\label{p:twostageboundX}
For any $\epsilon > 0$, there exists $N > 0$ such that 
\begin{equation}\label{e:LDlowerboundtwostages}
P_{x_n}( \sigma_1 < \tau_n < \tau_0) \le e^{-n (V_\sigma(0,x)-\epsilon)}, 
\end{equation}
for $n > N$, where $x_n = \lfloor nx \rfloor$, $0 < x(1) + x(2) < 1$,  $x \in {\mathbb R}_+^2.$
\end{proposition}
\begin{proof}
Parallel to the proof of Proposition \ref{p:LDlow},
we choose a sequence $\epsilon_n \rightarrow 0$ with $n\epsilon_n \rightarrow \infty$;
 \eqref{e:LDlowerboundtwostages} follows
from an application of
the optional sampling theorem to the supermartingale $M^{(n,\epsilon_n,\sigma)}$ and the
bound \eqref{e:toolongtime}.
\end{proof}

\subsection{LD limit for $P_x( \bar{\sigma}_1 < \tau < \infty )$}
\label{ss:LD3}
For this subsection and the next section it will be convenient to express
the $Y$ process in $x$ coordinates, we do this by setting,
$\bar{X}_k \doteq T_n(Y_k)$; $\bar{X}_k$ has the following dynamics:
\[
\bar{X}_{k+1} = \bar{X}_k + \pi_1(\bar{X}_k,  I_k).
\]
$\bar{\sigma}_1$ of \eqref{d:sigmas} in terms of $\bar{X}$ is 
$\bar{\sigma_1} = \inf\{k:\bar{X}_k \in \partial_1\}.$
The processes $\bar{X}$ and $X$ have the same dynamics except that $\bar{X}$
is not constrained on $\partial_1$. By definition, $\bar{X}_0 = X_0$.
Note the following: $\bar{X}$ hits $\{x \in {\mathbb Z} \times {\mathbb Z}_+:
x(1) + x(2) = n \}$ exactly when $Y$ hits $\{y \in {\mathbb Z} \times
{\mathbb Z}_+: y(1) = y(2)\}$; i.e., if we define
\[
\bar\tau_n \doteq \inf\{k: \bar{X}_k(1) + \bar{X}_k(2) = n \},
\]
then $\tau = \bar\tau_n.$
\begin{proposition}\label{p:twostageboundbarX}
For any $\epsilon > 0$, there exists $N > 0$ such that 
\begin{equation}\label{e:LDlowerboundtwostages}
P_{x_n}( \bar{\sigma}_1 < \bar\tau_n < \infty ) \le e^{-n (V_\sigma(0,x)-\epsilon)}, 
\end{equation}
for $n > N$, where $x_n = \lfloor nx \rfloor$, $ x(1) + x(2) < 1$,  $x \in {\mathbb R}\times {\mathbb R}_+$
\end{proposition}
\begin{proof}
The two stage subsolution $V(\cdot,\cdot,\cdot)$ of \eqref{d:defVsmooth2s}
is a subsolution for the $\bar{X}$ process as well because, $\bar{X}$
has identical dynamics as $X$ with one less constraint. Therefore,
the proof of Proposition \ref{p:twostageboundX}
applies verbatim to the current setup with one change: in the
proof of \eqref{e:LDlowerboundtwostages} we truncate time with
the bound \eqref{e:toolongtime} for $\tau_n$. We replace this with
the corresponding bound \eqref{e:boundontau} for  $\tau$.
\end{proof}

\section{Completion of the limit analysis}\label{s:together}
We now combine Propositions \ref{p:LDupper}, \ref{p:twostageboundX} and
\ref{p:twostageboundbarX} to get the main approximation result
of this work:
\begin{theorem}\label{t:mainapprox}
For any
$x \in {\mathbb R}_+^2$,
$x(1) + x(2) < 1$, $x(1) > 0$,
there exists $C_7 > 0$ and $N > 0$ such that
\[
\frac{ |P_{x_n}(\tau_n < \tau_0) - P_{T_n(x_n)}( \tau < \infty)|}{
P_{x_n}(\tau_n < \tau_0)}
 = 
\frac{ |P_{x_n}(\tau_n < \tau_0) - P_{x_n}(\bar{\tau}_n < \infty)|}{
P_{x_n}(\tau_n < \tau_0)}
< e^{-C_7 n}
\]
for $n > N$, where $x_n = \lfloor x n \rfloor$.
\end{theorem}
That
$P_{x_n}(\bar\tau_n < \infty) = P_{T_n(x_n)}(\tau < \infty)$
follows from the the definitions in subsection \ref{ss:LD3}.
\begin{proof}
The definitions \eqref{e:defV} and \eqref{d:defVsigma1} imply that
\[
2C_7= V_\sigma(x,0) - V(x) > 0
\]
for $x \in {\mathbb R}_+^2$, $x(1) + x(2)  < 1$, $x(1) > 0$.
Choose $\epsilon < C_7.$
The processes $X$ and $\bar{X}$ follow exactly the same path until they
hit $\partial_1$. It follows that
\begin{equation}\label{e:fas1}
|P_{x_n}(\tau_n < \tau_0) - P_{x_n}(\bar\tau_n < \infty)|
\le 
P_{x_n}(\sigma_1 < \tau_n < \tau_0) + P_{x_n}(\bar{\sigma}_1 < \bar\tau_n < \infty)
\end{equation}
By Propositions \ref{p:twostageboundX}, \ref{p:twostageboundbarX}
and \ref{p:LDupper}
there exists $N > 0$ such that
\begin{equation}\label{e:fas2}
P_{x_n}(\sigma_1 < \tau_n < \tau_0) + P_{x_n}(\bar{\sigma}_1 < \bar\tau_n < \infty)
\le
e^{-n(V_\sigma(0,x) -\epsilon/2)}
\end{equation}
and
\begin{equation}\label{e:fas3}
P_x(\tau_n <\tau_0) \ge e^{-n(V(x)-\epsilon/2)}
\end{equation}
for $n > N$.
The bounds \eqref{e:fas1}, \eqref{e:fas2} and \eqref{e:fas3}
give
\[
\frac{ |P_{x_n}(\tau_n < \tau_0) - P_{x_n}(\bar{\tau}_n < \infty)|}{
P_{x_n}(\tau_n < \tau_0)} < e^{-n C_7},
\]
for $n > N.$

\end{proof}

\section{Computation of $P_y( \tau < \infty)$}\label{s:Pty}
Theorem \ref{t:mainapprox} tells us that 
$P_{y}(\tau < \infty)$, $y = T_n(x_n)$, approximates
$P_{x_n}(\tau_n < \tau_0)$ with exponentially decaying relative error
for $x(1) > 0.$
To complete our analysis, it remains to compute $P_y(\tau < \infty)$.
As a function of $y$, $P_y(\tau < \infty)$ is a $Y$-harmonic function.
Furthermore, it is $\partial B$-determined, i.e., it has the representation
\[
y \rightarrow {\mathbb E}[ g(Y_\tau) 1_{\{\tau < \infty\}}]
\]
for some function $g$ on $\partial 
B$ (for $y\mapsto \partial P_y(\tau < \infty)$,
$g$ equals $1$ identically).
We will try to compute $P_y(\tau < \infty)$  as a superposition of the $Y$-harmonic functions expounded in Section 
\ref{s:Yharmonic}; because $P_y(\tau < \infty)$ is $1$ for $y \in \partial B$, we would like the
superposition to be as close to $1$ as possible on $\partial B$. 
We have two classes of $Y$-harmonic functions given
in Propositions \ref{p:both} 
(constructed from a single point on ${\mathcal H}$)
 and \ref{p:conjugate} (constructed from conjugate points on ${\mathcal H}$).
The first class gives us only one nontrivial $Y$-harmonic function, computed
in Lemma \ref{l:rho1}: $h_{\rho_1} =[(\rho_1,\rho_1),\cdot].$ Remember that we have assumed
$\bm\alpha(r,1) = r^2/\rho_2 < 1$. This implies that,
among the functions
in the second class, the most relevant for the computation of $P_y(\tau < \infty)$ is
\[
{\bm h}_r = \frac{1}{1-\rho_2/r}h_r =  [(r,1),\cdot] - \frac{1-r}{1-\rho_2/r} [ (r, \rho_2/r^2),\cdot],
\]
because this $Y$-harmonic function exponentially converges to $1$ for $y =(k,k) \in \partial B$ and
$k \rightarrow \infty.$
A simple criterion to check whether 
a $Y$-harmonic function of the form $\sum_{i=1}^I c_i[(\beta_i,\alpha_i)]$
is $\partial B$-determined is given in \cite{sezer2015exit}:
\begin{proposition}\label{p:pbdetermined}
A $Y$-harmonic function of the form $\sum_{i=1}^I c_i[(\beta_i,\alpha_i)]$ is $\partial B$ determined
if $|\beta_i|  < 1$ and $|\alpha_i| \le 1$, $i=1,2,3,...,I.$ 
\end{proposition}

Our first theoretical result on $P_y(\tau < \infty)$ 
arises from a linear combination
of $h_{\rho_1}$ and $h_r$:
\begin{proposition}\label{p:explicit}
If 
\begin{equation}\label{as:harmonic}
\rho_2\rho_1 = r^2 
\end{equation}
then
\begin{equation}\label{e:formula1}
P_y( \tau <\infty) = {\bm h}_r(y) + \frac{1-r}{1-\rho_2/r} h_{\rho_1}(y)
\end{equation}
for $y \in B.$
\end{proposition}
\begin{proof}
The right side of \eqref{e:formula1} is $Y$-harmonic by construction. Furthermore,
$\rho_2\rho_1 = r^2$ implies 
${\bm h}_r(y) + \frac{1-r}{1-\rho_2/r} h_{\rho_1}(y) = 1$
for $y \in \partial B.$ Therefore, to prove \eqref{e:formula1} it suffices
to prove that
\begin{equation}\label{e:formularight}
{\bm h}_r + \frac{1-r}{1-\rho_2/r} h_{\rho_1}
\end{equation}
is $\partial B$-determined. For this we will use Proposition \ref{p:pbdetermined};
in the present case,
the $\beta_i$ are $\rho_1, r < 1$ and the $\alpha_i$ are $1$ and $\rho_1 \le 1$. It follows
that \eqref{e:formularight} is $\partial B$-determined.
\end{proof}
If \eqref{as:harmonic} doesn't hold, i.e., if $r^2 \neq \rho_1 \rho_2$ then
one can proceed in several ways. 
As a first step, one can use the functions $h_r$ and
$h_{\rho_1}$ to construct lower and upper bounds on $P_y(\tau < \infty)$:
\begin{proposition}\label{p:relativeerr}
There exists positive constants $c_{0}$, $c_{1}$, and $C_8$
\begin{equation}\label{e:upperlower}
P_y( \tau < \infty)  \le h^{a,0}(y) \le C_8 P_y(\tau < \infty)
\end{equation}
where
\begin{equation}\label{d:defha}
h^{a,0} = c_0 {\bm h}_r + c_1 h_{\rho_1}.
\end{equation}
In particular, $h^{a,0}$ approximates $P_y(\tau < \infty)$ with bounded
relative error.
\end{proposition}
\begin{proof}
If $\rho_1 > r^2/\rho_2$, one can set $c_0 = 1$ and $c_1 = \frac{1-r}{1-\rho_2/r}$ since, for these values
\begin{equation}\label{e:boundonB}
h^{a,0} =  {\bm h}_r + \frac{1-r}{1-\rho_2/r} h_{\rho_1} \ge 1
\end{equation}
on $\partial B$. Both $h^{a,0}$ and $y\mapsto P_y( \tau < \infty)$;
this and \eqref{e:boundonB} imply $h^{a,0}(y) \ge P_y(\tau < \infty)$ for
$y \in B$. To get the second bound on \eqref{e:upperlower} set
\begin{equation}\label{d:C8}
C_8 =  1 + 
\frac{1-r}{1-\rho_2/r}
\max_{x \ge 0}\left[
 \rho_1^x - \left(\frac{r^2}{\rho_2}\right)^x \right].
\end{equation}
With this choice of $C_8$ we get the second bound in \eqref{e:upperlower}
on $\partial B$; that both $y \mapsto P_y(\tau < \infty)$ and
$h^{a,0}$ are $\partial B$-determined implies the same bound on all of
$B$.

If $\rho_1 < r^2/\rho_1$, first choose $C_0$ so that
\begin{equation}\label{e:boundhalf}
 1 + 
\min_{x \ge 0}\left[
C_0 \rho_1^x - 
\frac{1-r}{1-\rho_2/r}
\left(\frac{r^2}{\rho_2}\right)^x \right] \ge 1/2.
\end{equation}
Then
\[
h^{a,0}(y) = 2 h_r(y) + 2C_0 h_{\rho_1}(y) \ge 1
\]
for $y \in \partial B$, from which the first bound in \eqref{e:upperlower}
follows.
To get the second bound, set
\[
C_8/2 =  1 + 
\max_{x \ge 0}\left[
C_0 \rho_1^x - \frac{1-r}{1-\rho_2/r}
\left(\frac{r^2}{\rho_2}\right)^x \right],
\]
and proceed as above.
\end{proof}
Our choice of the constant $1/2$ in \eqref{e:boundhalf} is arbitrary,
any value between $(0,1)$ would suffice for the argument. Therefore,
the constants $c_0$ and $c_1$ are not unique and they can be optimized
to reduce relative error.

\begin{proposition}
For $x  \in {\mathbb R}_+^2$, $x(1) + x(2) < 1$, $x(1) > 0$, 
$x_n = \lfloor nx \rfloor$, and for $n$ large,
$h^{a,0}$ of \eqref{d:defha} evaluated at
$T_n(x_n)$ approximates $P_{x_n}(\tau_n < \tau_0)$ with
bounded relative error.
\end{proposition}
\begin{proof}
We know by Theorem \ref{t:mainapprox} that, for $x \in {\mathbb R}_+^2$,
 $x(1) + x(2) < 1$ and $x(1)> 0$,
$P_{T_n(x_n)}(\tau < \infty)$ approximates $P_{x_n}(\tau_n <\tau_0)$
with vanishing relative error. On the other hand,
the above Proposition tells us that $h^{a,0}$ of \eqref{d:defha}
approximates $P_y(\tau < \infty)$ with bounded relative error.
These imply that $h^{a,0}(T_n(x))$ approximates 
$P_{x_n}(\tau_n < \tau_0)$ with bounded
relative error.
\end{proof}

Proposition \ref{p:conjYz} gives not one but a one-complex-parameter family of
$Y$-harmonic functions. A natural question is whether one can obtain
finer approximations of $P_y(\tau < \infty)$ than what $h^{a,0}$ provides.
In this, we need $\partial B$-determined $Y$-harmonic functions. The
next proposition (an adaptation of \cite[Proposition 4.13]{sezer2015exit}
to the current setting) identifies a class of these which are naturally
suitable for the approximation
of $P_y(\tau < \infty).$
\begin{proposition}\label{p:pbdeterminedconj}
There exists $0 < R < 1$ such that for all $\alpha \in {\mathbb C}$
with $R < |\alpha| \le 1$ , $\max(|\beta_1(\alpha)| ,
|\bm\alpha(\beta_1(\alpha),\alpha)|) < 1$; in particular
$h_{\beta_1(\alpha)}$ is $\partial B$-determined.
\end{proposition}
\begin{proof}
We know by \cite[Proposition 4.7]{sezer2015exit} that 
$|\beta_1(\alpha)| \le r < 1$ for all $|\alpha|=1.$ Then
\[
\left|
\bm\alpha(\beta_1(\alpha),\alpha)
\right| = 
\left| \frac{\beta_1(\alpha)^2}{\alpha \rho_2}
\right| \le \frac{r^2}{\rho_2} < 1,
\]
where the last inequality is the assumption \eqref{as:alphaconj}.
The functions $\beta_1$ and $\bm \alpha$ are continuous; it follows that
the inequality above holds also for $R < |\alpha| \le 1$ if $R < 1$
is sufficiently close to $1$. That $h_{\beta_1}$ is $\partial B$-determined
follows from these and Proposition \ref{p:pbdetermined}.
\end{proof}

We can now use as many of the $\partial B$-determined $Y$-harmonic
functions identified in Propositions \ref{p:conjugate} and
\ref{p:pbdeterminedconj} as we like to
construct finer approximations of $P_y(\tau < \infty)$. Once the
approximation is constructed upperbounds on 
its relative error can be computed from
the maximum and the minimum of the approximation 
on $\partial B$- as was done in the proof of Proposition \ref{p:relativeerr}:
\begin{proposition}\label{p:improvedapprox}
Let $R$ be as in Proposition \ref{p:pbdetermined}.
For $c_k \in {\mathbb C}$ and
$R < |\alpha_k| \le 1$ 
$k=0,1,2,...,K$ 
define
\begin{equation}\label{e:improvedapprox}
h^{a,K} = \Re(h^{a*,K}), 
h^{a*,K} = {\bm h}_r + c_{0} h_{\rho_1} + \sum_{i=1}^K c_k h_{\beta_1(\alpha_k)}.
\end{equation}
Then $h^{a,K}$ is $Y$-harmonic and $\partial B$-determined.
Furthermore, for
\begin{equation}\label{d:defcs}
c^* = \max_{y \in \partial B} |h^{a*,K} -1| < \infty
\end{equation}
$h^{a*,K}$ approximates $P_y(\tau < \infty)$ with relative error bounded
by $c^*$.
\end{proposition}
\begin{proof}
We know by Propositions \ref{p:both} and \ref{p:conjugate} that
$h^{a*,K}$ is $Y$-harmonic. That $R < |\alpha_k| \le 1$ and
Proposition \ref{p:pbdetermined} imply that $h^{a*,K}$ is also
$\partial B$-determined, i.e., 
\[
h^{a*,K}(y) = {\mathbb E}_y[ h^{a*,K}(Y_{\tau}) 1_{\{\tau  < \infty\}}].
\]
Taking the real part of both sides gives:
\begin{equation}\label{e:haYdeter}
h^{a,K}(y) = {\mathbb E}_y[h^{a,K}(Y_{\tau}) 1_{\{\tau < \infty\}}],
\end{equation}
i.e, $h^{a,K}$ is $Y$-harmonic and $\partial B$-determined.
That $c^* < \infty$ follows from 
$\max|\beta_1(\alpha_k),\bm\alpha(\beta_1(\alpha_k),\alpha_k)| < 1$
(see Proposition \ref{p:pbdetermined}).
The inequality
\begin{equation}\label{e:boundonpartB}
1-c^* < h^{a,K}(k,k) <1 + c^*
\end{equation}
follows from \eqref{d:defcs},  $|\Re(z) -1| \le |z-1|$ 
for any $z \in {\mathbb C}.$
It follows from \eqref{e:boundonpartB} and \eqref{e:haYdeter} that
\begin{align*}
(1-c^*){\mathbb E}_y[1_{\{\tau < \infty\}}]
&\le h^{a,K}(y) \le 
(1+c^*){\mathbb E}_y[1_{\{\tau < \infty\}}]\\
(1-c^*)P_y( \tau < \infty)
&\le h^{a,K}(y) \le 
(1+c^*)P_y(\tau < \infty),
\end{align*}
This implies that $h^{a,K}$ approximates $P_y(\tau < \infty)$ with
relative error bounded by $c^*$.
\end{proof}
One of the key aspects of Proposition \ref{p:improvedapprox} is that
it shows us how to compute an upper bound on the relative error
of an approximation of the form \eqref{e:improvedapprox} from the values
it takes on $\partial B$. We can use this to choose the
$\alpha_k$ and the $c_k$ to reduce relative error, the next subsection
illustrates this procedure.
\subsection{Finer approximations when $r^2 \neq \rho_1 \rho_2$}\label{ss:finer}
To illustrate how one can use approximations of the form
\eqref{e:improvedapprox} to improve on the approximation provided
by Proposition \ref{p:relativeerr}, let us assign
values to the parameters
$\lambda_i$ and  $\mu_i$ satisfying the assumptions \eqref{as:orderrho},
\eqref{as:tech}:
\[
\lambda_1 = 0.1, \mu_1 = 0.2, \lambda_2 = 0.2, \mu_2 = 0.5;
\]
for these choice of parameters we have
\[
r = \frac{\lambda_1 +\lambda_2}{\mu_1 + \mu_2} = \frac{3}{7}.
\]
We note $r^2 = 9/49 \neq 1/5 = \rho_1 \rho_2$;
therefore, we don't have an explicit formula for $P_y(\tau < \infty).$
But Proposition \ref{p:relativeerr} implies that
\begin{equation}\label{e:ha0repeat}
h^{a,0}(y) ={\bm h}_r + \frac{1-r}{1-\rho_2/r}h_{\rho_1}
\end{equation}
 approximates
$P_y(\tau < \infty)$ with relative error bounded by
\begin{align*}
C_8 - 1 &= \frac{1-r}{1-\rho^2/r}\left( \rho_1^{x^*} - 
\alpha_2^{x^*} \right) = 0.3607,\\
\alpha_2 &= \frac{r^2}{\rho_2}, ~~
x^* =  \log(\log(\rho_1)/\log(\alpha_2))/ (\log(\alpha_1) - \log(\rho_1))
\end{align*}
where $C_8$ is computed as in \eqref{d:C8}. Then, by Theorem
\ref{t:mainapprox}, $h^{a,0}(T_n(x_n))$ approximates $P_{x_n}(\tau_n > \tau_0)$
with relative error converging to a level bounded by $C_8 -1 = 0.3607$ .

We can reduce this error
by using further $Y$-harmonic functions given by Propositions
\ref{p:conjugate} and \ref{p:pbdeterminedconj} and constructing
an approximation of the form \eqref{e:improvedapprox}.
We note $\beta_1(0.7)=0.34610$ and therefore, by an argument parallel
to the proof of Proposition \ref{p:pbdeterminedconj}, we infer that
 $|\beta_1(\alpha)| \le 0.34619 $ ,
 $|\bm(\beta_1(\alpha),\alpha)| < 1$
for $|\alpha|=0.7$. Thus, $h_{\beta_1(\alpha)}$ is
 $Y$-harmonic and $\partial B$-determined 
for all $|\alpha|=0.7$, and we can use this class of functions
in improving our approximation of $P_y(\tau < \infty).$
Let us begin with using
$K=3$ additional $Y$-harmonic functions of this form in our approximation:
for the $\alpha$'s let us take
\[
\alpha_{1,j} = 0.7 e^{i \frac{j}{4} 2\pi}, j\in \{1,2,3=K\}.
\]
The resulting harmonic functions are
\[
h_{\beta_1(\alpha_{1,j})}, j \in \{1,2,3=K\}.
\]
(see \eqref{e:betas} and \eqref{d:hbeta}).

Our approximation $h^{a,K}$ will be of the form 
\[
h^{a,K} = \Re(h^{a*,K}),~~
h^{a*,K} = {\bm h}_r + c_{1,0} h_{\rho_1} + \sum_{j=1}^K c_{1,j} h_{\beta_1(\alpha_{1,j})}.
\]
One can choose the coefficients $c_{1,j}$, $j\in \{0,1,2,3=K\}$ in a number
of ways, for example, by
minimizing $L_p$ errors. Here we will proceed in the following simple way:
the ideal situation would be $h^{a,K}(y) = 1$ for all $y \in \partial B$, which
would mean $h^{a,K}(y) = P_y(\tau < \infty)$, but this will not hold in 
general. We will instead
require that this identity holds for $y=(k,k)$, $k=0,1,2,3$. This leads
to the following four dimensional linear equation:
\begin{equation}\label{e:projection}
1 = h^{a,K}(k,k) = {\bm h}_r((k,k)) + c_{1,0} h_{\rho_1}(k,k) + \sum_{j=1}^3 c_{1,j} h_{\beta_1(\alpha_{1,j})}(k,k),
\end{equation}
$k=0,1,2,3$;
Solving \eqref{e:projection} gives
\begin{align*}
c_{1,0} &=  7.80744 - 0.12974i  ,~~~  c_{1,1} =-0.25880 + 1.46155i   \\
c_{1,2} &= -0.26358 - 0.01349i
,~~~c_{1,3}= 0.17597 + 0.01433i
\end{align*}
Once the approximation is computed, following Proposition \ref{p:improvedapprox}
one can easily compute its
relative error in approximating $P_y(\tau < \infty)$ by computing
\[
\max_{y \in \partial B} |h^{a*,K}(y) - 1|.
\]
That
$\max_{j=1,2,..,K}(r,\rho_1, r^2/\rho_1,|\alpha_{1,j}|, |\beta_1(\alpha_{1,j}),
\bm{\alpha}(\beta_1(\alpha_{1,j}),\alpha_{1,j})) < 1$
implies \\
$\argmax_{ y \in \partial B} |h^{a*,K}(y) - 1|$
is finite. For $h^{a*,K}$ computed above, the maximizer turns out to be
$y^* = (4,4)=(K+1,K+1)$ and the maximum approximation error is
\begin{equation}\label{e:cstar}
c^* = \max_{ y \in \partial B} |h^{a*,K}(y) - 1| = |h^{a,K}((y^*)) - 1| = 0.1136.
\end{equation}
The graph of the approximation
error $|h^{a*,K}(y)-1|$, $y \in \partial B$ is shown in Figure \ref{f:approximationerror1}.
\begin{figure}[h]
\centering
\hspace{-1cm}
\begin{subfigure}{0.48\textwidth}
\scalebox{0.45}{
\includegraphics{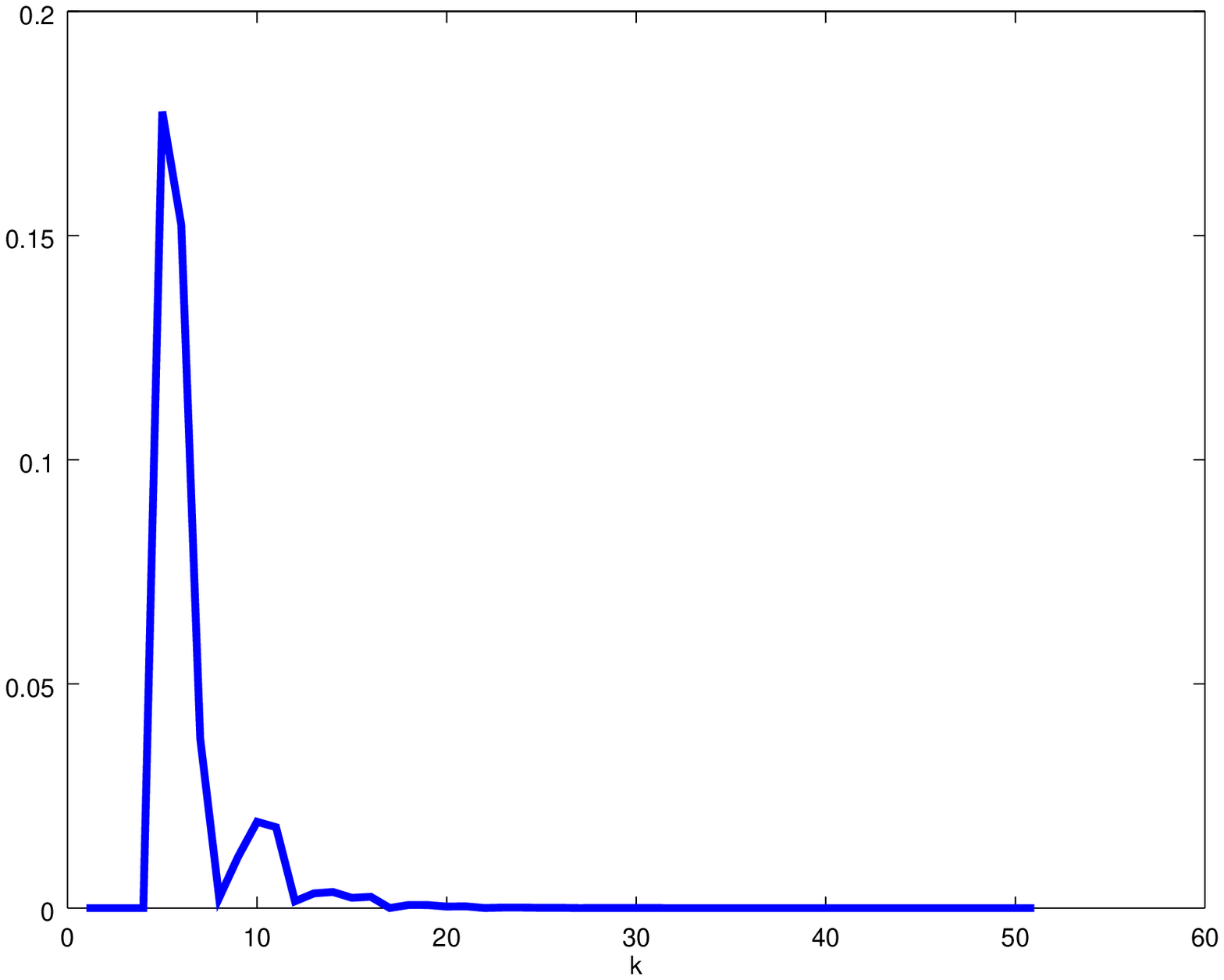}}
\caption{ $|h^{a*,K}(y)-1|$ as a function of $y=(k,k)$}
\label{f:approximationerror1}
\end{subfigure}
\hspace{0.5cm}
\begin{subfigure}{0.48\textwidth}
\scalebox{0.45}{
\includegraphics{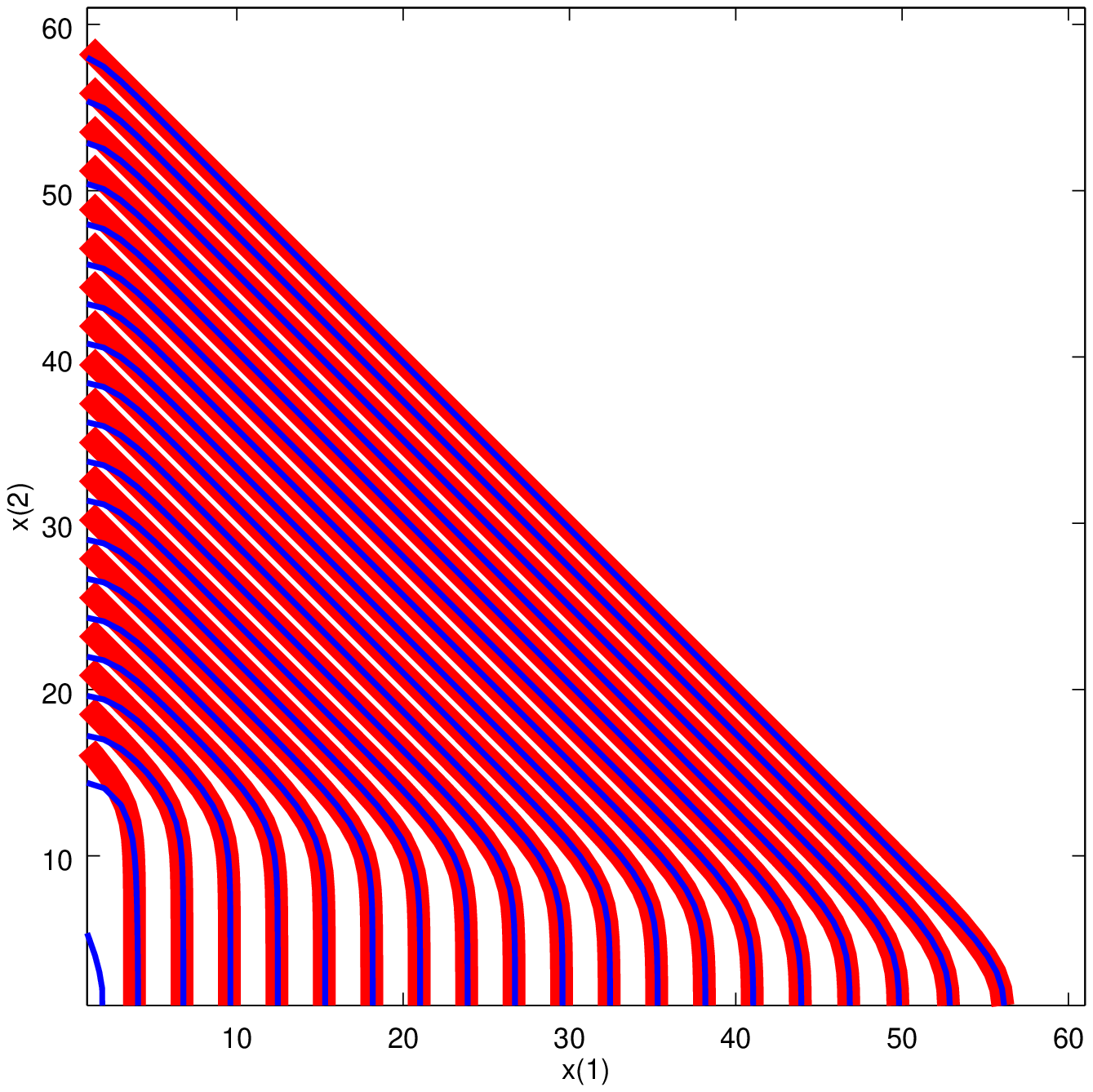}}
\caption{Level curves of $-\frac{1}{n} \log h^{a,0}(T_n(x))$ and 
$-\frac{1}{n}\log P_x(\tau_n < \tau_0)$, $n=60$}
\label{f:approx1}
\end{subfigure}
\end{figure}

By Proposition \ref{p:improvedapprox}
\[
\left|\frac{P_y( \tau < \infty) - h^{a,0}(y)}{P_y(\tau < \infty)}
\right| \le c^*.
\]
Theorem \ref{t:mainapprox} now implies that $h^{a,K}(T_n(x_n))$ approximates
$P_{x_n}(\tau_n < \tau_0)$ with relative error bounded by
$c^* = 0.17764$ for $n$ large.
Therefore, in improving our approximation from $h^{a,0}$ of
\eqref{e:ha0repeat} to $h^{a,3}$ by adding three $Y$-harmonic functions
of the form $h_{\beta_1(\alpha_{1,j})}$  to the approximating  basis,
the relative error decreases from $c_0^* = 0.3607$ to $c_0^* = 0.17764$.
Figure \ref{f:approx1} shows the level curves of 
$-\frac{1}{n} \log h^{a,3}(T_n(x))$ and 
$-\frac{1}{n}\log P_x(\tau_n < \tau_0)$ (the latter computed numerically
via iteration of the harmonic equation satisfied by $P_x(\tau_n < \tau_0)$)
for $n=60$;
the level curves overlap completely except along
$\partial_1$, as suggested by our analysis.

To illustrate how the approximation error decreases when $K$ increases,
let us repeat the computation above with $K=20$. The resulting
maximum relative error turns out to be:
\[
c^* = \max_{ y \in \partial B} |h^{a*,20}(y) - 1| = |h^{a*,20}((21,21)) - 1| = 
1.6211 \times 10^{-3}.
\]
The probability $P_{(4,0)}(\tau_{60} < \tau_0)$, computed numerically,
equals $4.6658 \times 10^{-17}$, the best approximation of this
quantity computed above is
$h^{a,20}(56,0) = 5.2 \times 10^{-17}.$ The discrepancy arises from
the proximity of $(4,0)$ to $\partial_1.$ As we move away from the
$\partial_1$, these quantities get closer
$P_{(10,0)}(\tau_{60} < \tau_0) = 3.3303 \times 10^{-15}$,
$h^{a,20}(50,0) =  3.3358 \times 10^{-15}$, compatible with
the maximum relative error computed above.
Figure \ref{f:convergence} shows how approximation improves
as $K$ increases:
\ninseps{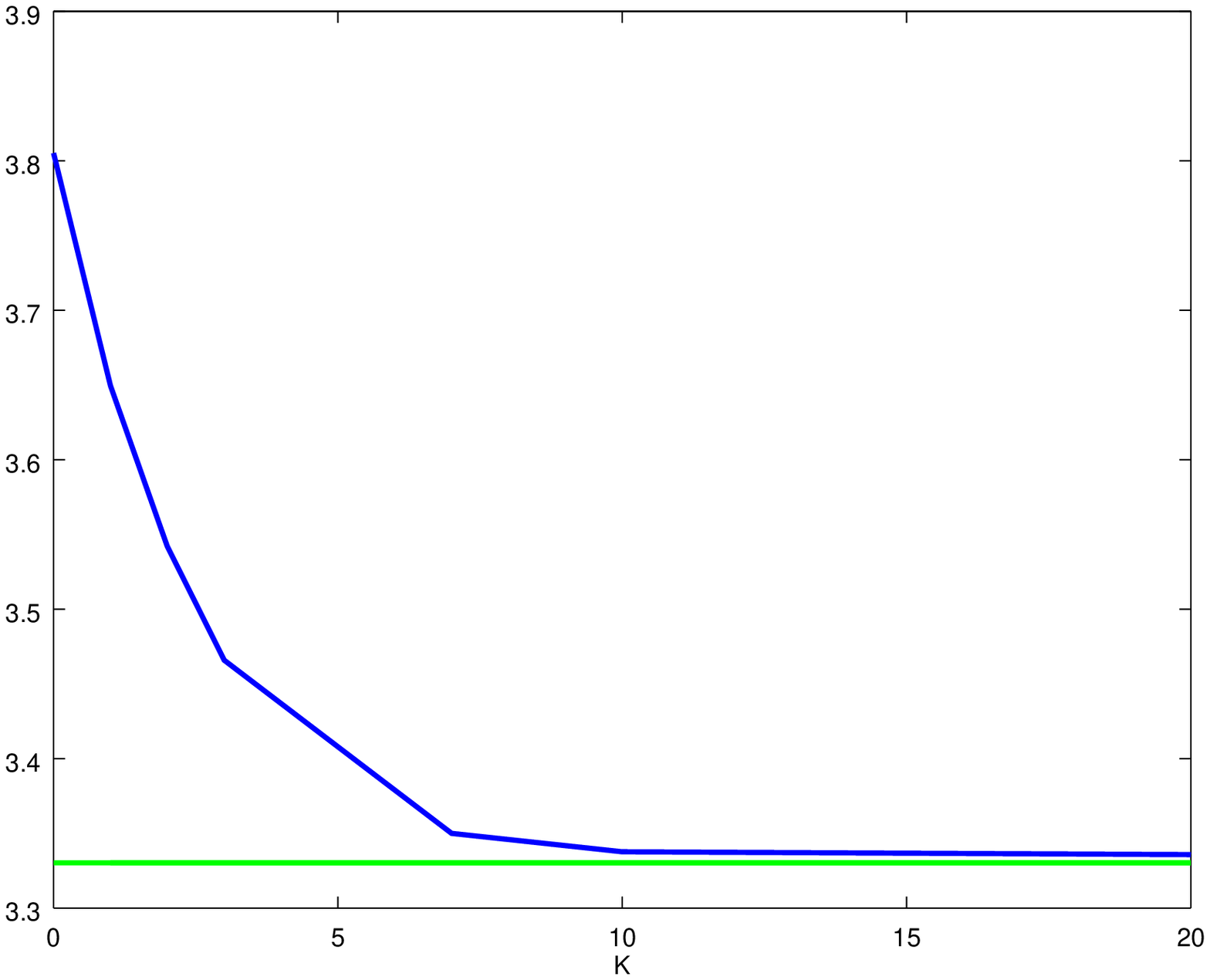}{$K \mapsto h^{a,K}(50,0)$ and $P_{(10,0)}(\tau_{60} < \tau_0)$
(the flat line), drawn at $10^{-15}$ scale}{0.45}
\vspace{-0.5cm}
\section{Comparison with the tandem case}\label{s:compare}
This section compares the analysis and results of the current work
to those of \cite{tandemsezer} treating the approximation
of the probability $P_{x}(\tau_n < \tau_0)$ for the
 constrained random walk representing two tandem queues, 
which has the increments $(1,0)$, $(-1,1)$ and $(0,-1).$ The main
idea is the same for both walks: i.e., approximation of
$P_x(\tau_n < \tau_0)$ by $P_y(\tau < \infty)$ and computing/ approximating
the latter via harmonic functions constructed out of single and
conjugate points on the
characteristic surface. However, the assumptions, the results and the analysis
manifest nontrivial differences. 
Let us begin with the assumptions:
\paragraph{Assumption $r^2/\rho_2 < 1$}
In the tandem case $\beta_1(1) = \rho_2$ and
the conjugate point of $(\rho_2,1)$ is $(\rho_2,\rho_1)$,
therefore, the stability assumption automatically implies
$\bm\alpha(r,1) < 1$. For the parallel case, $\bm\alpha(r,1)$ can indeed
be greater than $1$ if $r$ and $\rho_1$ are close and $\rho_2$ is small;
we therefore explicitly assume $r^2/\rho_2 < 1.$ This assumption appears
in two places: 1) in the convergence analysis, in the derivation of the bound
\eqref{e:toprovelaplace} and 2) in the computation of $P_y(\tau <\infty)$
in Section \ref{s:Pty}. We think that the use of the assumption 
$r^2/\rho_2 < 1$
in the first case can be removed without much change from the arguments
of the present and earlier works; the details remain for future work.
We think that
the computation of $P_y(\tau < \infty)$ when
$r^2 /\rho_2 > 1$ presents genuine difficulties,
the treatment of which also remains for future work.
Next we point out the differences in results:
\paragraph{Region where $P_y(\tau < \infty)$ is a good approximation
for $P_x(\tau_n < \tau_0)$}
That the tandem walk involves no jumps of the form $(-1,0)$
implies that
$P_{T_n(x_n)}(\tau < \infty)$ provides an approximation of
$P_{x_n}(\tau_n < \tau_0)$ with exponentially decaying
relative error for all $x$ away from $0$;
in contrast, the presence of the jump $(-1,0)$ in the parallel case,
implies that the same
approximation works only away from $\partial_1$ for the
parallel walk case treated in the present work. This difference
shows itself in the proofs of exponential
decay of relative error, too, this is discussed below.

\paragraph{Explicit formula for $P_y(\tau < \infty)$}
In the case of the tandem walk, the probability $P_y(\tau  <\infty)$
can be explicitly represented as a linear combination 
of the harmonic functions
$h_{\rho_1}$ and $h_{\rho_2}$ for all stable parameter values
as long as $\mu_1 \neq \mu_2$; in the parallel case this only happens
when $r^2 = \rho_1 \rho_2$ (see Proposition \ref{p:explicit}).
When $r^2 \neq \rho_1 \rho_2$, $h_r $ and $h_{\rho_1}$ can only provide
an approximation of $P_y(\tau < \infty)$ with bounded relative
error (Proposition \ref{p:relativeerr}). This relative error can be
reduced by adding into the approximation further $\partial B$-determined
$Y$-harmonic functions (Proposition \ref{p:improvedapprox} and
subsection \ref{ss:finer}). 

The changes in argument from the tandem walk to the parallel walk
are as follows:
\paragraph{Analysis of $P_x(\tau_n < \tau_0)$}
In prior works \cite{thesis, DSW, sezer2009importance, sezer2015exit} 
the LD analysis of $P_x(\tau_n < \tau_0)$
and similar quantities
are based on sub and supersolutions of the limit HJB equation, similar to
the analysis given in subsection \ref{ss:ldlow}. In the present work, a novelty
is the use of explicit subharmonic functions (Proposition \ref{p:subharmonicpn})
of the constrained random walk $X$
in the proof of the upperbound Proposition \ref{p:LDupper}.
\paragraph{Analysis of $P_x( \sigma_1 < \tau_n < \tau_0)$}
The probability corresponding to
$P_x (\sigma_1 < \tau_n < \tau_0)$ in the tandem case
is $P_x(\sigma_1 < \sigma_{1,2} < \tau_n < \tau_0)$. 
For the proof of the exponential decay of the relative error,
we need upperbound on these probabilities. Both papers develop
these upperbound from subsolutions to a limit HJB equation.
The subsolution consists of three
pieces (one for each of the stopping times $\sigma_1$, $\sigma_{1,2}$
and $\tau_n$) for the tandem walk, and two pieces for the parallel walk
(one for each of the times $\sigma_1$ and $\tau_n$). In the tandem case,
the pieces of the subsolution are constructed from the subsolution for
the probability $P_x(\tau_n < \tau_0)$, whereas in the parallel case
a new piece is introduced based on the gradient $r_4$ of
\eqref{d:r4}.
\paragraph{Analysis of $P_x (\bar{\sigma}_1 < \tau < \infty)$}
The probability corresponding to
$P_x (\bar{\sigma}_1 < \tau < \infty)$ in the tandem case
is $P_x(\bar{\sigma}_1 < \bar{\sigma}_{1,2} < \tau < \infty)$.
The special nature of the tandem walk allowed us to find upperbounds
on this probability from the explicit formula we have
for $P_y(\tau < \infty)$; this significantly simplified the analysis
of the tandem walk case. For the parallel walk, we extended the
analysis of $P_x(\sigma_1 < \tau_n < \tau_0)$, based on subsolutions,
to
$P_x (\bar{\sigma}_1 < \tau < \infty)$. In this, the most significant
novelty is the analysis given Section \ref{s:laplace}, where we
prove the existence of $z > 1$ such that 
${\mathbb E}_z[ z^\tau 1_{\{\tau < \infty\}}] < \infty$. 
For this, we introduce what we call $Y-z$-harmonic
functions and provide methods of construction of classes of them
from points on $1/z$-level characteristic surfaces, which are
generalizations of characteristic surfaces.

\section{Conclusion}\label{s:conclusion}
 The probability $P_y(\tau < \infty)$ approximates
$P_{x}(\tau_n < \tau_0)$ well when $x$ is away from $\partial_1$;
as noted in the previous section, this is in contrast to the tandem case,
where the approximation is good away from the origin.
How can one extend the approximation to the region along $\partial_1$?
A natural idea, already pointed out in \cite{sezer2015exit} is to 
repeat the same analysis, but this time taking the corner 
$(0,n)$ as the origin of the $Y$ process, i.e., to use the change of
coordinate $y =T_n(x) = (x(1),n-x(2))$ to construct the $Y$ process.
Numerical calculations
indicate
that the resulting approximation will be accurate (i.e., exponentially
decaying relative error) along $\partial_1$
between the points $(0,n)$  
and $(0, \lfloor (1-C_4)n \rfloor)$ (see \eqref{d:C4}
for the definition of $C_4$).
 We believe that arguments and computations
parallel to the ones given in the present work would imply these results;
the details are left for future work. We think that
the extension of the approximation
to the region along the line
segment between $(0,0)$ and $(0, \lfloor (1-c_1)n
\rfloor )$  requires further ideas and computations.

We expect the analysis linking $P_{x}(\tau_n < \tau_0)$ to 
$P_y(\tau < \infty)$ when $\rho_1 = \rho_2$
to be parallel to the analysis given in the current work.
For the computation of $P_y(\tau < \infty)$, when $\rho_1 = \rho_2$,
the case $\lambda_1 = \lambda_2$, $\mu_1 = \mu_2$ appears to be
particularly simple. In this case, upon taking limits in \eqref{e:formula1}
one obtains
\[
P_y(\tau < \infty) = r^{y(1) - y(2)}+ (1-r)r^{y(1)}
( y(1)-y(2)),
\]
where $r=\rho_1 = \rho_2$.
A complete analysis of the computation of
$P_y(\tau < \infty)$ when $\rho_1 = \rho_2$ remains for future
work.

In subsection \ref{ss:finer},
the computation of $P_y(\tau < \infty)$ when $r^2\neq \rho_1 \rho_2$
proceeds as follows: 1) we first construct a candidate approximation
$h^{a,K} = \Re(h^{a*,K})$
of $P_y(\tau < \infty)$ 2) we find an upperbound on the relative
error of the approximation by finding the maximum of
$|h^{a*,K} -1|$ on $\partial B.$
A natural question is the following: given a relative error bound,
can we know apriori that an approximation having that maximum relative
error can be constructed? If that is possible, how many
$Y$-harmonic functions of the form given by Proposition \ref{p:conjugate}
would we need? To answer these questions require a fine
understanding of the functional analytic properties of
the span of the $\partial B$-determined $Y$-harmonic
functions given by Propositions \ref{p:both} and \ref{p:conjugate}.
This appears to be a difficult problem because the functions
given in these propositions don't have simple geometric properties,
such as the orthogonality of the Fourier basis in $L^2$.
A study of this problem remains for future work.

The exact formula for $P_y(\tau < \infty)$ 
for the tandem case has a remarkable
extension to $d$ dimensions; this is derived in 
\cite{sezer2015exit} and is based on harmonic-systems, a concept
defined in that work.
We think that it is also possible, in the case of parallel queues,
to obtain nontrivial harmonic systems
in higher dimensions. A complete characterization of such systems
and the question of under what conditions
they would give a rich class of $Y$-harmonic functions to approximate
$P_y(\tau < \infty)$ also remain challenging problems for future research.

{\small
\bibliography{balayage}
}

\end{document}

%% file: dyn2tex.tex
\begin{picture}(0,0)%
\includegraphics{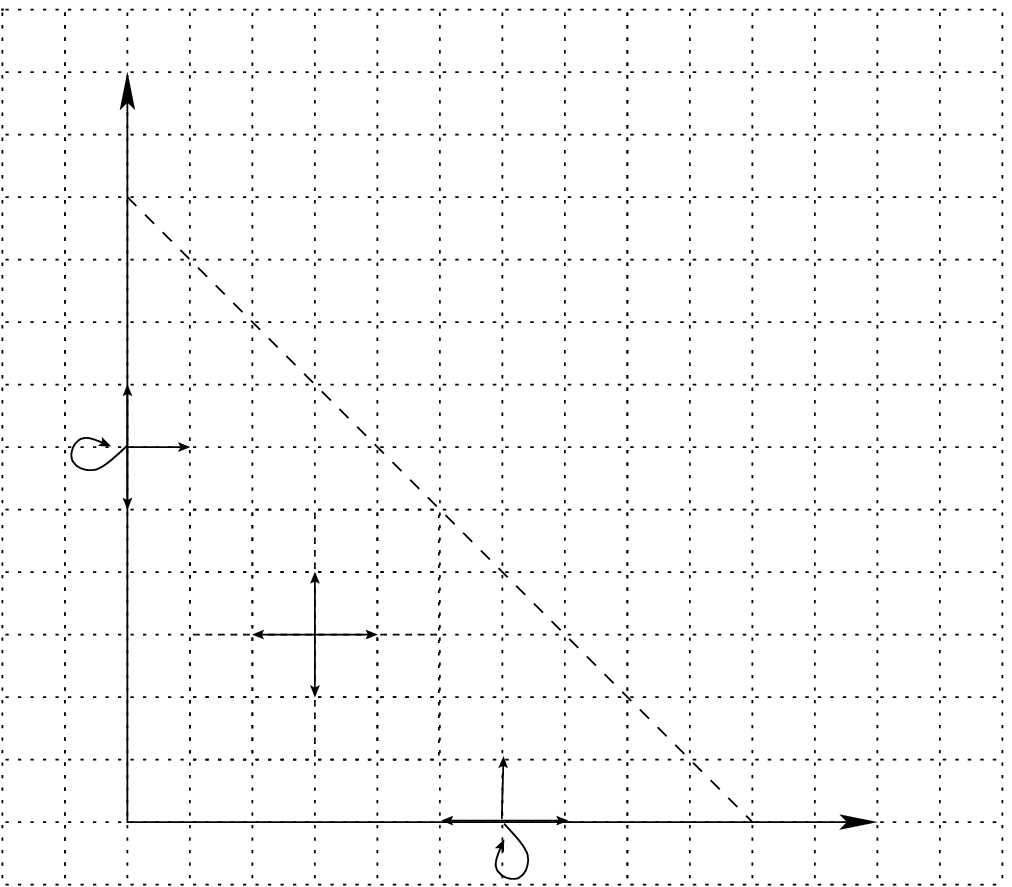}%
\end{picture}%
\setlength{\unitlength}{3947sp}%
\begingroup\makeatletter\ifx\SetFigFont\undefined%
\gdef\SetFigFont#1#2#3#4#5{%
  \reset@font\fontsize{#1}{#2pt}%
  \fontfamily{#3}\fontseries{#4}\fontshape{#5}%
  \selectfont}%
\fi\endgroup%
\begin{picture}(4824,4224)(5389,-5473)
\put(7211,-4174){\makebox(0,0)[lb]{\smash{{\SetFigFont{12}{14.4}{\rmdefault}{\mddefault}{\updefault}{\color[rgb]{0,0,0}$\lambda_1$}%
}}}}
\put(5963,-5341){\makebox(0,0)[lb]{\smash{{\SetFigFont{12}{14.4}{\rmdefault}{\mddefault}{\updefault}{\color[rgb]{0,0,0}$0$}%
}}}}
\put(6864,-3943){\makebox(0,0)[lb]{\smash{{\SetFigFont{12}{14.4}{\rmdefault}{\mddefault}{\updefault}{\color[rgb]{0,0,0}$\lambda_2$}%
}}}}
\put(6864,-4746){\makebox(0,0)[lb]{\smash{{\SetFigFont{12}{14.4}{\rmdefault}{\mddefault}{\updefault}{\color[rgb]{0,0,0}$\mu_2$}%
}}}}
\put(7426,-3586){\makebox(0,0)[lb]{\smash{{\SetFigFont{12}{14.4}{\rmdefault}{\mddefault}{\updefault}{\color[rgb]{0,0,0}$\partial A_n$}%
}}}}
\put(6376,-4176){\makebox(0,0)[lb]{\smash{{\SetFigFont{12}{14.4}{\rmdefault}{\mddefault}{\updefault}{\color[rgb]{0,0,0}$\mu_1$}%
}}}}
\end{picture}%

%% file: trntex.tex
\begin{picture}(0,0)%
\includegraphics{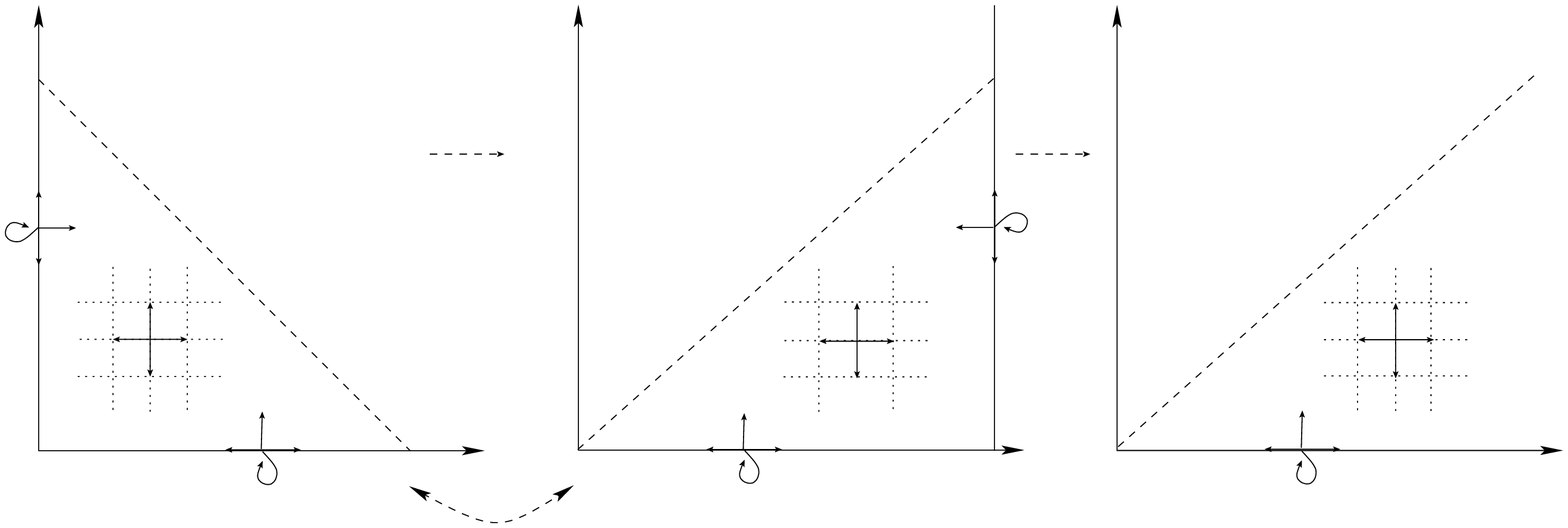}%
\end{picture}%
\setlength{\unitlength}{3947sp}%
\begingroup\makeatletter\ifx\SetFigFont\undefined%
\gdef\SetFigFont#1#2#3#4#5{%
  \reset@font\fontsize{#1}{#2pt}%
  \fontfamily{#3}\fontseries{#4}\fontshape{#5}%
  \selectfont}%
\fi\endgroup%
\begin{picture}(12601,4210)(5720,-5754)
\put(16126,-3586){\makebox(0,0)[lb]{\smash{{\SetFigFont{12}{14.4}{\rmdefault}{\mddefault}{\updefault}{\color[rgb]{0,0,0}$\partial B$}%
}}}}
\put(5963,-5341){\makebox(0,0)[lb]{\smash{{\SetFigFont{12}{14.4}{\rmdefault}{\mddefault}{\updefault}{\color[rgb]{0,0,0}$0$}%
}}}}
\put(10306,-5341){\makebox(0,0)[lb]{\smash{{\SetFigFont{12}{14.4}{\rmdefault}{\mddefault}{\updefault}{\color[rgb]{0,0,0}$0$}%
}}}}
\put(6864,-3943){\makebox(0,0)[lb]{\smash{{\SetFigFont{12}{14.4}{\rmdefault}{\mddefault}{\updefault}{\color[rgb]{0,0,0}$\lambda_2$}%
}}}}
\put(12556,-3929){\makebox(0,0)[lb]{\smash{{\SetFigFont{12}{14.4}{\rmdefault}{\mddefault}{\updefault}{\color[rgb]{0,0,0}$\lambda_2$}%
}}}}
\put(6864,-4746){\makebox(0,0)[lb]{\smash{{\SetFigFont{12}{14.4}{\rmdefault}{\mddefault}{\updefault}{\color[rgb]{0,0,0}$\mu_2$}%
}}}}
\put(12555,-4754){\makebox(0,0)[lb]{\smash{{\SetFigFont{12}{14.4}{\rmdefault}{\mddefault}{\updefault}{\color[rgb]{0,0,0}$\mu_2$}%
}}}}
\put(7164,-5616){\makebox(0,0)[lb]{\smash{{\SetFigFont{12}{14.4}{\rmdefault}{\mddefault}{\updefault}{\color[rgb]{0,0,0}$X$}%
}}}}
\put(8941,-5339){\makebox(0,0)[lb]{\smash{{\SetFigFont{12}{14.4}{\rmdefault}{\mddefault}{\updefault}{\color[rgb]{0,0,0}$ne_n$}%
}}}}
\put(7426,-3586){\makebox(0,0)[lb]{\smash{{\SetFigFont{12}{14.4}{\rmdefault}{\mddefault}{\updefault}{\color[rgb]{0,0,0}$\partial A_n$}%
}}}}
\put(6376,-4176){\makebox(0,0)[lb]{\smash{{\SetFigFont{12}{14.4}{\rmdefault}{\mddefault}{\updefault}{\color[rgb]{0,0,0}$\mu_1$}%
}}}}
\put(7211,-4174){\makebox(0,0)[lb]{\smash{{\SetFigFont{12}{14.4}{\rmdefault}{\mddefault}{\updefault}{\color[rgb]{0,0,0}$\lambda_1$}%
}}}}
\put(12937,-4183){\makebox(0,0)[lb]{\smash{{\SetFigFont{12}{14.4}{\rmdefault}{\mddefault}{\updefault}{\color[rgb]{0,0,0}$\mu_1$}%
}}}}
\put(12074,-4186){\makebox(0,0)[lb]{\smash{{\SetFigFont{12}{14.4}{\rmdefault}{\mddefault}{\updefault}{\color[rgb]{0,0,0}$\lambda_1$}%
}}}}
\put(13651,-5311){\makebox(0,0)[lb]{\smash{{\SetFigFont{12}{14.4}{\rmdefault}{\mddefault}{\updefault}{\color[rgb]{0,0,0}$ne_n$}%
}}}}
\put(17257,-4236){\makebox(0,0)[lb]{\smash{{\SetFigFont{12}{14.4}{\rmdefault}{\mddefault}{\updefault}{\color[rgb]{0,0,0}$\mu_1$}%
}}}}
\put(16461,-4201){\makebox(0,0)[lb]{\smash{{\SetFigFont{12}{14.4}{\rmdefault}{\mddefault}{\updefault}{\color[rgb]{0,0,0}$\lambda_1$}%
}}}}
\put(16906,-3892){\makebox(0,0)[lb]{\smash{{\SetFigFont{12}{14.4}{\rmdefault}{\mddefault}{\updefault}{\color[rgb]{0,0,0}$\lambda_2$}%
}}}}
\put(16913,-4769){\makebox(0,0)[lb]{\smash{{\SetFigFont{12}{14.4}{\rmdefault}{\mddefault}{\updefault}{\color[rgb]{0,0,0}$\mu_2$}%
}}}}
\put(13854,-2607){\makebox(0,0)[lb]{\smash{{\SetFigFont{12}{14.4}{\rmdefault}{\mddefault}{\updefault}{\color[rgb]{0,0,0}$n \rightarrow \infty$}%
}}}}
\put(9301,-2611){\makebox(0,0)[lb]{\smash{{\SetFigFont{12}{14.4}{\rmdefault}{\mddefault}{\updefault}{\color[rgb]{0,0,0}$T_n$}%
}}}}
\put(11979,-5623){\makebox(0,0)[lb]{\smash{{\SetFigFont{12}{14.4}{\rmdefault}{\mddefault}{\updefault}{\color[rgb]{0,0,0}${Y}^n$}%
}}}}
\put(16501,-5604){\makebox(0,0)[lb]{\smash{{\SetFigFont{12}{14.4}{\rmdefault}{\mddefault}{\updefault}{\color[rgb]{0,0,0}$Y$}%
}}}}
\put(11701,-3586){\makebox(0,0)[lb]{\smash{{\SetFigFont{12}{14.4}{\rmdefault}{\mddefault}{\updefault}{\color[rgb]{0,0,0}$\partial B_n$}%
}}}}
\end{picture}%

%% file: charsurf0.tex
\setlength{\unitlength}{1pt}
\begin{picture}(0,0)
\includegraphics{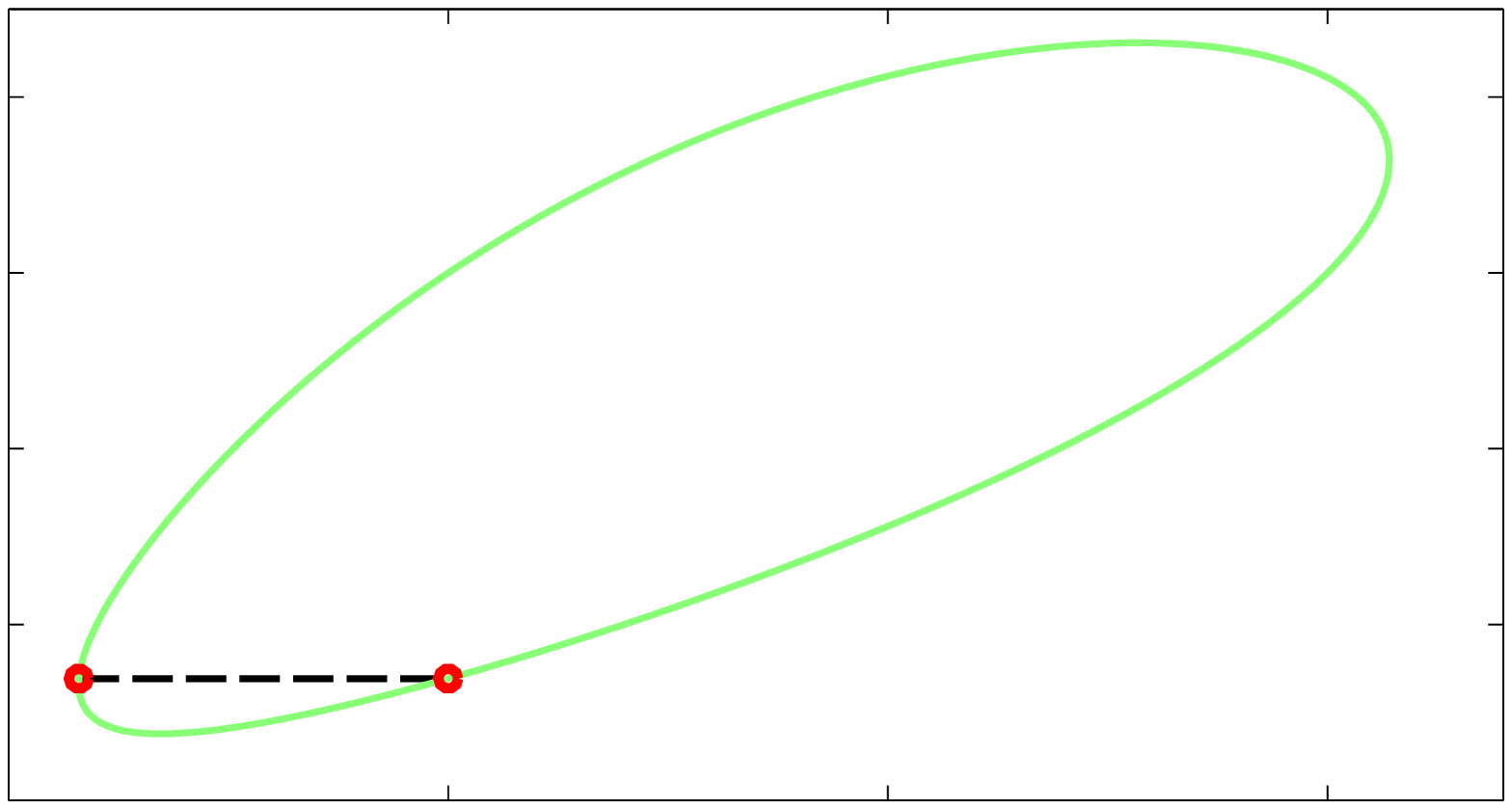}
\end{picture}%
\begin{picture}(576,432)(0,0)
\fontsize{12}{0}
\selectfont\put(74.88,100.388){\makebox(0,0)[t]{\textcolor[rgb]{0,0,0}{{0.5}}}}
\fontsize{12}{0}
\selectfont\put(206.174,100.388){\makebox(0,0)[t]{\textcolor[rgb]{0,0,0}{{1}}}}
\fontsize{12}{0}
\selectfont\put(337.468,100.388){\makebox(0,0)[t]{\textcolor[rgb]{0,0,0}{{1.5}}}}
\fontsize{12}{0}
\selectfont\put(468.762,100.388){\makebox(0,0)[t]{\textcolor[rgb]{0,0,0}{{2}}}}
\fontsize{12}{0}
\selectfont\put(69.8755,105.395){\makebox(0,0)[r]{\textcolor[rgb]{0,0,0}{{0.4}}}}
\fontsize{12}{0}
\selectfont\put(69.8755,157.913){\makebox(0,0)[r]{\textcolor[rgb]{0,0,0}{{0.6}}}}
\fontsize{12}{0}
\selectfont\put(69.8755,210.431){\makebox(0,0)[r]{\textcolor[rgb]{0,0,0}{{0.8}}}}
\fontsize{12}{0}
\selectfont\put(69.8755,262.948){\makebox(0,0)[r]{\textcolor[rgb]{0,0,0}{{1}}}}
\fontsize{12}{0}
\selectfont\put(69.8755,315.466){\makebox(0,0)[r]{\textcolor[rgb]{0,0,0}{{1.2}}}}
\fontsize{12}{0}
\selectfont\put(298.08,89.3883){\makebox(0,0)[t]{\textcolor[rgb]{0,0,0}{{$\alpha$}}}}
\fontsize{12}{0}
\selectfont\put(50.8755,223.56){\rotatebox{90}{\makebox(0,0)[b]{\textcolor[rgb]{0,0,0}{{$\beta$}}}}}
\end{picture}

%% file: ldrate.tex
\setlength{\unitlength}{1pt}
\begin{picture}(0,0)
\includegraphics{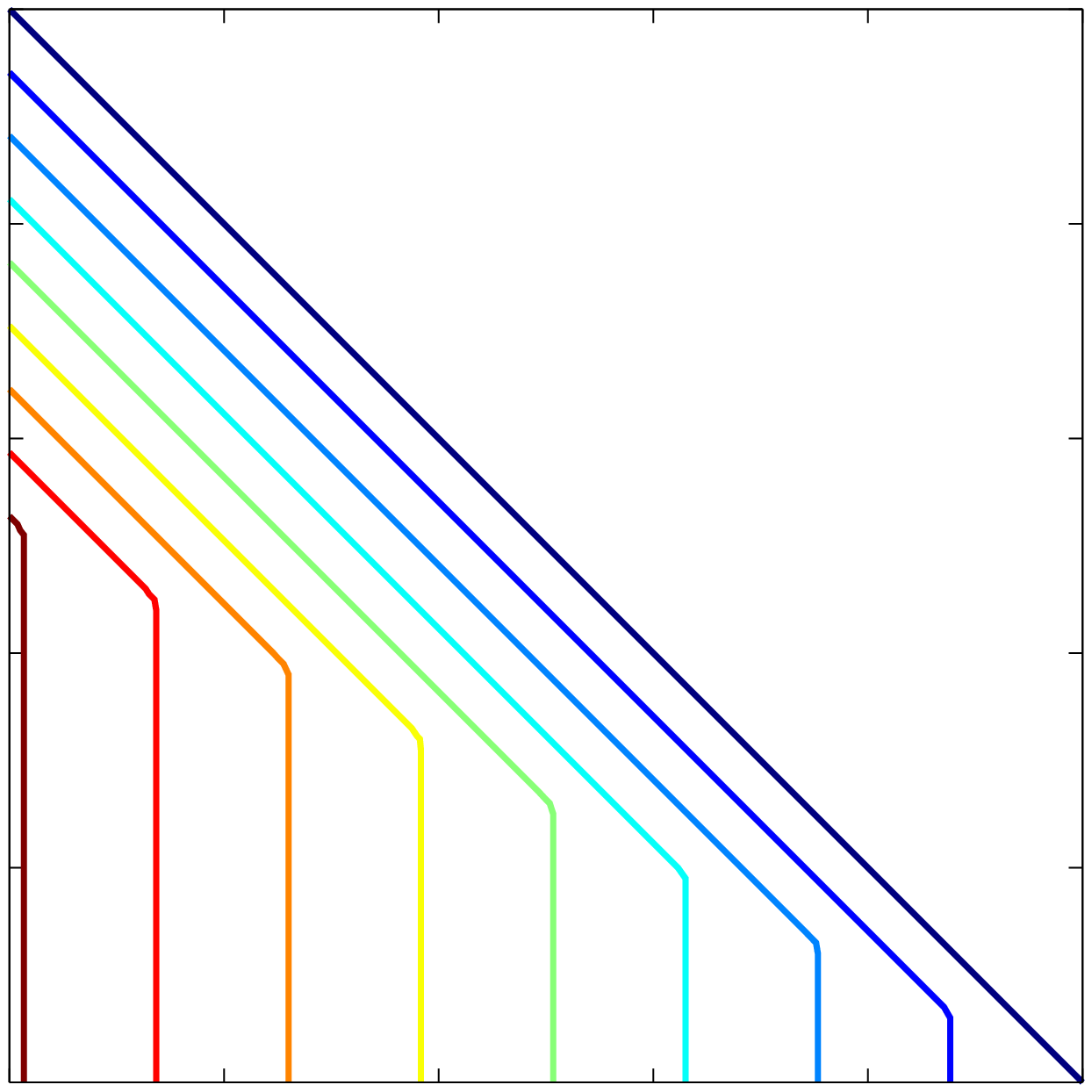}
\end{picture}%
\begin{picture}(576,432)(0,0)
\fontsize{15}{0}
\selectfont\put(122.04,42.5189){\makebox(0,0)[t]{\textcolor[rgb]{0,0,0}{{0}}}}
\fontsize{15}{0}
\selectfont\put(192.456,42.5189){\makebox(0,0)[t]{\textcolor[rgb]{0,0,0}{{0.2}}}}
\fontsize{15}{0}
\selectfont\put(262.872,42.5189){\makebox(0,0)[t]{\textcolor[rgb]{0,0,0}{{0.4}}}}
\fontsize{15}{0}
\selectfont\put(333.288,42.5189){\makebox(0,0)[t]{\textcolor[rgb]{0,0,0}{{0.6}}}}
\fontsize{15}{0}
\selectfont\put(403.704,42.5189){\makebox(0,0)[t]{\textcolor[rgb]{0,0,0}{{0.8}}}}
\fontsize{15}{0}
\selectfont\put(474.12,42.5189){\makebox(0,0)[t]{\textcolor[rgb]{0,0,0}{{1}}}}
\fontsize{15}{0}
\selectfont\put(117.039,47.52){\makebox(0,0)[r]{\textcolor[rgb]{0,0,0}{{0}}}}
\fontsize{15}{0}
\selectfont\put(117.039,117.936){\makebox(0,0)[r]{\textcolor[rgb]{0,0,0}{{0.2}}}}
\fontsize{15}{0}
\selectfont\put(117.039,188.352){\makebox(0,0)[r]{\textcolor[rgb]{0,0,0}{{0.4}}}}
\fontsize{15}{0}
\selectfont\put(117.039,258.768){\makebox(0,0)[r]{\textcolor[rgb]{0,0,0}{{0.6}}}}
\fontsize{15}{0}
\selectfont\put(117.039,329.184){\makebox(0,0)[r]{\textcolor[rgb]{0,0,0}{{0.8}}}}
\fontsize{15}{0}
\selectfont\put(117.039,399.6){\makebox(0,0)[r]{\textcolor[rgb]{0,0,0}{{1}}}}
\fontsize{15}{0}
\selectfont\put(298.08,31.5189){\makebox(0,0)[t]{\textcolor[rgb]{0,0,0}{{$x(1)$}}}}
\fontsize{15}{0}
\selectfont\put(98.0389,223.56){\rotatebox{90}{\makebox(0,0)[b]{\textcolor[rgb]{0,0,0}{{$x(2)$}}}}}
\end{picture}

%% file: charsurf.tex
\setlength{\unitlength}{1pt}
\begin{picture}(0,0)
\includegraphics{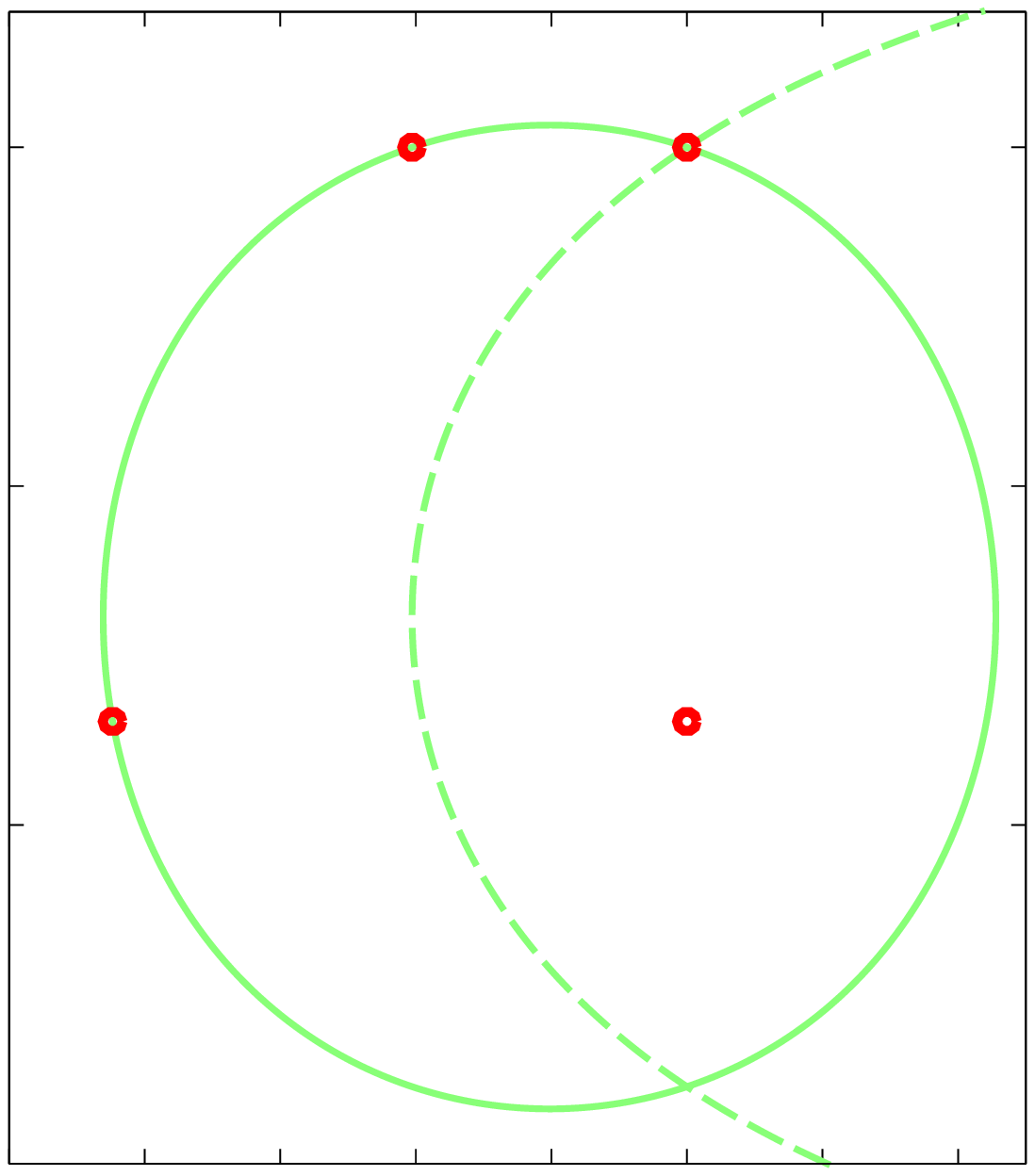}
\end{picture}%
\begin{picture}(576,432)(0,0)
\fontsize{15}{0}
\selectfont\put(142.751,42.5189){\makebox(0,0)[t]{\textcolor[rgb]{0,0,0}{{-1}}}}
\fontsize{15}{0}
\selectfont\put(184.172,42.5189){\makebox(0,0)[t]{\textcolor[rgb]{0,0,0}{{-0.8}}}}
\fontsize{15}{0}
\selectfont\put(225.593,42.5189){\makebox(0,0)[t]{\textcolor[rgb]{0,0,0}{{-0.6}}}}
\fontsize{15}{0}
\selectfont\put(267.014,42.5189){\makebox(0,0)[t]{\textcolor[rgb]{0,0,0}{{-0.4}}}}
\fontsize{15}{0}
\selectfont\put(308.435,42.5189){\makebox(0,0)[t]{\textcolor[rgb]{0,0,0}{{-0.2}}}}
\fontsize{15}{0}
\selectfont\put(349.856,42.5189){\makebox(0,0)[t]{\textcolor[rgb]{0,0,0}{{0}}}}
\fontsize{15}{0}
\selectfont\put(391.278,42.5189){\makebox(0,0)[t]{\textcolor[rgb]{0,0,0}{{0.2}}}}
\fontsize{15}{0}
\selectfont\put(432.699,42.5189){\makebox(0,0)[t]{\textcolor[rgb]{0,0,0}{{0.4}}}}
\fontsize{15}{0}
\selectfont\put(137.756,47.52){\makebox(0,0)[r]{\textcolor[rgb]{0,0,0}{{-1.5}}}}
\fontsize{15}{0}
\selectfont\put(137.756,151.073){\makebox(0,0)[r]{\textcolor[rgb]{0,0,0}{{-1}}}}
\fontsize{15}{0}
\selectfont\put(137.756,254.626){\makebox(0,0)[r]{\textcolor[rgb]{0,0,0}{{-0.5}}}}
\fontsize{15}{0}
\selectfont\put(137.756,358.179){\makebox(0,0)[r]{\textcolor[rgb]{0,0,0}{{0}}}}
\fontsize{15}{0}
\selectfont\put(298.08,31.5189){\makebox(0,0)[t]{\textcolor[rgb]{0,0,0}{{$q_1$}}}}
\fontsize{15}{0}
\selectfont\put(114.756,223.56){\rotatebox{90}{\makebox(0,0)[b]{\textcolor[rgb]{0,0,0}{{$q_2$}}}}}
\end{picture}

%% file: regionstex.tex
\begin{picture}(0,0)%
\includegraphics{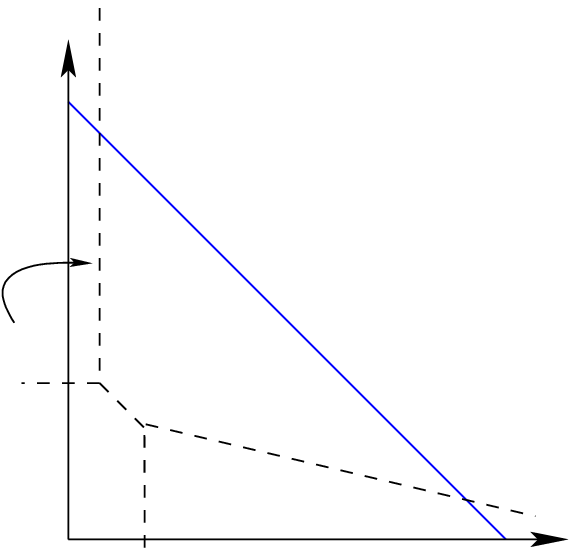}%
\end{picture}%
\setlength{\unitlength}{3947sp}%
\begingroup\makeatletter\ifx\SetFigFont\undefined%
\gdef\SetFigFont#1#2#3#4#5{%
  \reset@font\fontsize{#1}{#2pt}%
  \fontfamily{#3}\fontseries{#4}\fontshape{#5}%
  \selectfont}%
\fi\endgroup%
\begin{picture}(2740,2796)(2073,-2995)
\put(2251,-736){\makebox(0,0)[lb]{\smash{{\SetFigFont{12}{14.4}{\rmdefault}{\mddefault}{\updefault}{\color[rgb]{0,0,0}$1$}%
}}}}
\put(4446,-2926){\makebox(0,0)[lb]{\smash{{\SetFigFont{12}{14.4}{\rmdefault}{\mddefault}{\updefault}{\color[rgb]{0,0,0}$1$}%
}}}}
\put(2091,-1876){\makebox(0,0)[lb]{\smash{{\SetFigFont{12}{14.4}{\rmdefault}{\mddefault}{\updefault}{\color[rgb]{0,0,0}$R_2$}%
}}}}
\put(2401,-2536){\makebox(0,0)[lb]{\smash{{\SetFigFont{12}{14.4}{\rmdefault}{\mddefault}{\updefault}{\color[rgb]{0,0,0}$R_0$}%
}}}}
\put(3076,-2611){\makebox(0,0)[lb]{\smash{{\SetFigFont{12}{14.4}{\rmdefault}{\mddefault}{\updefault}{\color[rgb]{0,0,0}$R_1$}%
}}}}
\put(2926,-1861){\makebox(0,0)[lb]{\smash{{\SetFigFont{12}{14.4}{\rmdefault}{\mddefault}{\updefault}{\color[rgb]{0,0,0}$R_3$}%
}}}}
\end{picture}%